\documentclass[onefignum, onetabnum]{siamart171218}
\usepackage{caption}
\usepackage{subcaption}
\usepackage{amsmath, amssymb, amsopn}
\usepackage{array}
\usepackage{amscd}
\usepackage{epsfig}
\usepackage{color}
\usepackage{comment}
\usepackage{setspace}
\usepackage{textcomp}
\usepackage{stmaryrd}
\usepackage{multirow}
\usepackage{diagbox}
\usepackage{todonotes}
\usepackage{subcaption}
\usepackage{placeins}
\usepackage{xcolor}
\usepackage{mathtools}
\usepackage{multirow}
\usepackage{xcolor}
\usepackage{mathdots}
\usepackage{longtable}
\usepackage{booktabs}

\numberwithin{equation}{section}

\newtheorem{thm}[theorem]{Theorem}

\newtheorem{example}[theorem]{Example}
\newtheorem{rmk}[theorem]{Remark}

\DeclareMathOperator{\Div}{div}

\DeclareMathOperator{\Grad}{\nabla}

\DeclareMathOperator{\ssum}{\textstyle \sum}
\DeclareMathOperator{\diag}{diag}

\newcommand{\inner}[2]{\langle #1, #2 \rangle}
\newcommand{\foralls}{\forall \,}

\newcommand{\R}{{\mathbb R}}
\newcommand{\strain}{\varepsilon}

\headers{Robust preconditioner for MPET}{E. Piersanti, J. J. Lee, K.-A. Mardal, M. E. Rognes, and T. Thompson}

\title{Parameter robust preconditioning by congruence for multiple-network poroelasticity \thanks{Submitted to the editors \today. \funding{The work of J.~J.~Lee has been supported by the European Research Council under the European Union's Seventh Framework Programme (FP7/2007-2013) ERC grant agreement 339643. The work of K.-A. Mardal, T. Thompson and M.~E.~Rognes have been supported by the Research Council of Norway under the FRINATEK Young Research Talents Programme through project \#250731/F20 (Waterscape). E.~Piersanti is a doctoral fellow in the Simula-UCSD-University of Oslo Research and PhD training (SUURPh) program, an international collaboration in computational biology and medicine funded by the Norwegian Ministry of Education and Research.}}}

\author{E. Piersanti\thanks{Simula Research Laboratory, P. O. Box 134, 1325 Lysaker, Norway (\email{eleonora@simula.no})}
  \and J. J. Lee\thanks{Department of Mathematics, Baylor University, One Bear Place \# 97328, Waco, Texas 78798, USA (\email{jeonghun\_lee@baylor.edu})}
 \and T.~Thompson \thanks{Simula Research Laboratory, Norway. Currently at: Mathematical Institute, Oxford University (\email{thompsont@maths.ox.ac.uk})}
  \and K.-A. Mardal\thanks{Department of Mathematics, University of Oslo, P. O. Box 1053 Blindern, 0316 Oslo, Norway and Simula Research Laboratory, P. O. Box 134, 1325 Lysaker, Norway (\email{kent-and@simula.no})}
  \and M. E. Rognes\thanks{Simula Research Laboratory, P. O. Box 134, 1325 Lysaker, Norway (\email{meg@simula.no})}}

\begin{document}

\maketitle
\begin{abstract}
  The mechanical behaviour of a poroelastic medium permeated by
  multiple interacting fluid networks can be described by a system of
  time-dependent partial differential equations known as the
  multiple-network poroelasticity (MPET) equations or
  multi-porosity/multi-permeability systems. These equations
  generalize Biot's equations, which describe the mechanics of the
  one--network case. The efficient numerical solution of the MPET
  equations is challenging, in part due to the complexity of the
  system and in part due to the presence of interacting parameter
  regimes. In this paper, we present a new strategy for efficiently
  and robustly solving the MPET equations numerically. In particular,
  we introduce a new approach to formulating finite element methods
  and associated preconditioners for the MPET equations. The approach
  is based on designing transformations of variables that
  simultaneously diagonalize (by congruence) the equations' key
  operators and subsequently constructing parameter-robust
  block-diagonal preconditioners for the transformed system. Our
  methodology is supported by theoretical considerations as well as
  by numerical results.
\end{abstract}

\section{Introduction}

In this paper, we consider the preconditioned iterative solution of
finite element discretizations of the multiple-network poroelasticity
(MPET) equations. These equations traditionally originate in
geomechanics where they are also known under the term
multi-porosity/multi-permeability systems~\cite{BaiEtAl1993}. The MPET
equations generalize Biot's equations~\cite{Biot1941} from the one
network to the multiple network case, and multi-compartment Darcy
(MPT) equations~\cite{MichlerEtAl2013} from a porous (but rigid) to a
poroelastic medium. Over the last decade, the MPT and MPET equations
have seen a surge of interest in biology and physiology; e.g.~to model
perfusion in the heart~\cite{MichlerEtAl2013, LeeEtAl2015}, cancer~\cite{shipley2020four}
brain~\cite{JozsaEtAl2019} or~liver~\cite{BravsnovaEtAl2018}, or to
model the interaction between elastic deformation and fluid flow and
transport in the brain~\cite{ChouEtAl2016, EisentragerEtAl2013,
  MokhtarudinEtAl2017, TullyEtAl2013, TullyVentikos2011,
  VardakisEtAl2016}.

Concretely, the quasi--static MPET equations read as
follows~\cite{BaiEtAl1993}: for a given number of networks $J \in
\mathbb{N}$, find the displacement $u$ and the network pressures
$p_{j}$ for $j = 1, \dots, J$ such that
\begin{subequations}
  \label{eq:mpet}
  \begin{align}
    \label{eq:mpet1} 
    - \Div (2 \mu \strain(u) +\lambda \Div u \Bbb{I}) + \ssum_{j} \alpha_{j} \Grad p_{j} &= f, \\
    \label{eq:mpet2}    
    s_{j} \dot{p}_{j} + \alpha_{j} \Div \dot{u} - \Div K_{j} \Grad p_{j} + \ssum_i \xi_{j \leftarrow i} (p_j - p_i) &= g_{j},
  \end{align}
\end{subequations}
where $u = u(x, t)$, $p_{j} = p_{j}(x, t)$ for $x \in \Omega
\subset \R^{d}$ ($d = 1, 2, 3$), $t \in (0, T]$, and $\Bbb{I}$ is the $d \times d$ identity matrix. Physically,
  the equations \eqref{eq:mpet} describe a porous and elastic medium
  permeated by a number of fluid networks under the assumptions that
  the solid matrix can be modeled as isotropic and linearly elastic
  with Lam\'e constants $\mu > 0$ and $\lambda > 0$, and the transfer
  between the networks is regulated by the corresponding pressure
  differences with exchange coefficients $\xi_{j \leftarrow i} \geq
  0$. For each network $j$, we define the Biot-Willis coefficient
  $\alpha_{j} \in (0, 1]$ such that $\sum_j \alpha_{j} \leq 1$, the
    storage coefficient $s_{j} > 0$, and the hydraulic conductivity
    tensor $K_{j} = \kappa_{j}/\nu_{j} > 0$ with $\kappa_{j}$ and $\nu_{j}$ being
    the permeability and fluid viscosity, respectively. Moreover, $\Grad$ denotes
    the column-wise gradient, $\strain$ is the symmetric gradient, $\Div$
    denotes the (row-wise) divergence the superposed dot denotes the
    time derivative(s), and $I$ denotes the identity matrix. On the
    right hand side, $f$ represents body forces and $g_{j}$ sources
    (or sinks) in network $j$ for $j = 1, \dots, J$.

The MPT equations represents a reduced version
of~\eqref{eq:mpet} that result from ignoring the elastic contribution
of the solid matrix. These equations then read as follows: for a given
number of networks $J \in \mathbb{N}$, find the network pressures
$p_{j}$ for $j = 1, \dots, J$ such that
\begin{equation}
  \label{eq:mpt}
  - \Div K_{j} \Grad p_{j} + \ssum_{i=1}^J \xi_{j \leftarrow i} (p_j - p_i) = g_{j},
\end{equation}
where for $i, j = 1, \dots, J$, $p_{j} = p_{j}(x)$ for $x \in \Omega
\subset \R^{d}$ ($d = 1, 2, 3$), the parameters $K_j$ and $\xi_{j
  \leftarrow i}$ remain the hydraulic conductivity and exchange
coefficients, respectively, and $g_j$ again represents other sources
(or sinks) in each network.

\begin{center}
\begin{table}
  \begin{tabular}{lcrc}
    \toprule
    Parameter  & Unit & Value & Reference \\
    \midrule
    Hydraulic conductivities ($K_j$) & mm$^2$ (kPa s)$^{-1}$ & & \\
    \midrule
    Brain gray matter & & $2.0 \times 10^{-3}$ &  \cite{StoverudEtAl2016}\\
    Brain white matter & & $2.0 \times 10^{-2}$ &  \cite{StoverudEtAl2016} \\
    Cardiac arteries & & $1.0$ & \cite{MichlerEtAl2013} \\
    Cardiac capillaries & & $2.0$ & \cite{MichlerEtAl2013} \\
    Cardiac veins & & $10.0$ & \cite{MichlerEtAl2013} \\
    Brain vasculature & & $3.75 \times 10^{1}$ & \cite{VardakisEtAl2016} \\
    Brain fluid exchange & & $1.57 \times 10^{-2}$ & \cite{VardakisEtAl2016} \\
    \midrule
    Exchange coefficients ($\xi_{j \rightarrow i}$) & (kPa s)$^{-1}$ & & \\
    \midrule
    Brain capillary-vasculature & & $1.5 \times 10^{-16}$ & \cite{VardakisEtAl2016} \\
    Brain capillary-tissue fluid & & $2.0 \times 10^{-16}$ &\cite{VardakisEtAl2016} \\
    Brain tissue fluid-veins & & $2.0 \times 10^{-10}$ &\cite{VardakisEtAl2016} \\
    Cardiac capillary-arteries & & $2.0 \times 10^{-2}$ & \cite{MichlerEtAl2013} \\  
    Cardiac capillary-veins & & $5.0 \times 10^{-2}$ & \cite{MichlerEtAl2013} \\  
    \bottomrule
  \end{tabular}
  \caption{Sample parameter values for hydraulic conductivities and
    exchange coefficients with reference to~\eqref{eq:mpet}
    and/or~\eqref{eq:mpt}.}
  \label{tab:parameters}
\end{table}
\end{center}

The relative size of the conductivities $K_j$ and the exchange 
coefficients $\xi_{j\leftarrow i}$ may vary tremendously in applications.  
Large parameter variation is certainly present in applied problems of a 
physiological nature; a selection of representative parameter values, from 
research literature, is given in Table~\ref{tab:parameters}.  

Here, we see that the hydraulic conductivities span four orders of magnitude while
the exchange coefficients span fourteen orders of magnitude. Hence,
there is a need for preconditioners that are robust with respect
to variations in parameters.  Physiological 
applications, in particular, can benefit from preconditioners which are robust 
with respect to $K_j$, $\xi_{j\leftarrow i}$ and $\lambda$ as in \eqref{eq:mpet} 
and \eqref{eq:mpt}. 

With this in mind, parameter-robust numerical approximations and
solution algorithms for~\eqref{eq:mpet} is currently an active
research topic. In the nearly incompressible case $\lambda \gg 1$, the
standard two-field variational formulation of~\eqref{eq:mpet} is not
robust. To address this challenge, we introduced and analyzed a mixed
finite element method for the MPET equations based on a total pressure
formulation in \cite{lee2019mixed}. We note that the total pressure
in case of one network was presented in~\cite{oyarzua2016locking,LeeEtAl2017}.
Hong et al.~\cite{hong2019conservative} shortly thereafter extended 
the three-field formulation in \cite{hong2018classical} to MPET equations 
taking the displacement, the network fluid fluxes and the network pressures 
as unknowns and targeting a range of parameter
regimes. As an alternative to these fully coupled approaches a form of
splitting schemes has been analyzed by Lee~\cite{Lee2019}. Regarding
the iterative solution and preconditioning of the fully coupled
formulations, a robust preconditioner for Biot's equations (the case
for $J=1$) was presented by Lee et al.~\cite{LeeEtAl2017}. Hong et
al.~\cite{hong2019conservative} presented both theoretical results and
numerical examples regarding parameter-robust preconditioners for the MPET 
equations with their extended three-field-type formulation.
Hong et al. further developed parameter-robust solver algorithms, an iterative solver algorithm 
using the iterative coupling approach (cf. \cite{MikelicWheeler2013}) 
in~\cite{hong2018parameter}, and an Uzawa-type algorithm in~\cite{hong2019parameter}.

In this paper, we present a new approach for preconditioning linear
systems of equations resulting from a conforming finite element 
discretization of the total pressure variational formulation of the MPET 
equations. The key idea, as introduced for the MPT equations
in~\cite{PiersantiEtAlEnumath2019}, to design a parameter-dependent
transformation of the pressure variables $p = (p_1, \dots, p_J)$ into
a set of transformed variables $\tilde p$. The transformation should
be such that the originally coupled exchange operator decouples while
the originally decoupled diffusion operator stays decoupled
(diagonal). The design of such a transform hinges on the concept of
\emph{diagonalization by congruence} and associated matrix
theory. After transformation, we then define a natural block diagonal
preconditioner for the transformed system of equations. This strategy
yields a parameter--robust preconditioner, which we both prove
theoretically and demonstrate numerically.

This manuscript is organized as follows. We introduce notation and
review relevant preconditioning and matrix theory in
Section~\ref{sec:preliminaries}. We briefly consider the reduced case
of the MPT equations in Section~\ref{sec:mpt} before turning to the
analysis of the preconditioner for the MPET equations in
Section~\ref{sec:mpet}. Finally, we present some conclusions and
outlook in Section~\ref{sec:conclusion}.

\section{Preliminaries}
\label{sec:preliminaries}
%

In Section~\ref{sec:preconditioning}
we briefly review preconditioning of parameter-dependent 
systems and state a known result regarding simultaneous 
diagonalization by congruence.  Notation for the remainder of the manuscript 
is discussed in Section~\ref{sec:notation}.

\subsection{Notation}
\label{sec:notation}

In the subsequent manuscript, we use the following notation. Let
$\Omega$ be an open, bounded domain in $\mathbb{R}^d$, $d = 2, 3$,
with Lipschitz polyhedral boundary $\partial \Omega$. We denote by $L^2(\Omega)$
the space of square integrable functions on $\Omega$ with inner
product $\inner{\cdot}{\cdot}$ and norm $\| \cdot \|$. We denote by
$H^m(\Omega)$ the standard Sobolev space with norm $\lVert \cdot
\rVert_{H^m}$ and semi-norm $| \cdot |_{H^m}$ for $m \geq 1$ and
$H^m(\Omega; \R^d)$ the corresponding $d$-vector fields. We use
$H^m_0(\Omega)$ to denote the subspace of $H^m(\Omega)$ with vanishing trace
on the boundary of $\Omega$. Let $\Gamma$ be a subset of $\partial \Omega$ such that 
$\partial \Omega \setminus \Gamma$ has a positive $(d-1)$-dimensional Lebesgue measure. 
$H_{\Gamma}^m(\Omega)$ is the subspace of $H^m(\Omega)$ such that the elements in 
$H_{\Gamma}^m(\Omega)$ have vanishing trace on $\Gamma$. $H_{\Gamma}^m(\Omega; \R^d)$ 
is the subspace of $H^m(\Omega ; \R^d)$ such that every $v_j$ in 
$(v_1, \ldots, v_d) \in H^m(\Omega; \R^d)$ is an element in $H_{\Gamma}^m(\Omega)$. 
Throughout this paper we set $\Gamma$ a fixed subset of $\partial \Omega$ satisfying the aforementioned assumption.

We introduce the parameter-dependent $L^2$-inner product and norm:
\begin{equation*}
 \| p \|_{\beta}^2 = \inner{p}{p}_{\beta} = \inner{\beta p}{p}
\end{equation*}
for $\beta \in L^{\infty}(\Omega)$, $\beta(x) > 0$, and $p \in
L^2(\Omega)$ (and similarly for vector or tensor fields).  The notation 
$\mathbb{I}$ will denote an identity $d \times d$ matrix while $I_V$ will denote the 
identity operator on a Hilbert space $V$.

To be self-contained we recall the Kronecker product of matrices. If $A$ in $\mathbb{R}^{m\times n}$ 
and $B\in \mathbb{R}^{r\times s}$ are two real-valued matrices then 
$A\otimes B$ is the $mr\times ns$ matrix defined by multiplying each entry of 
$A$ by the matrix $B$.  That is,
\begin{align} \label{eq:kronecker}
A\otimes B = \left[ \begin{array}{cccc} 
a_{11} B & a_{12} B & \dots & a_{1n} B \\  
a_{21} B & a_{22} B & \dots & a_{2n} B \\
\vdots   & \ddots  & \ddots & \vdots  \\
a_{m 1} B & a_{m 2} B & \dots & a_{mn} B
\end{array}\right].
\end{align}
We can consider its natural extension for a matrix $A$ and a linear operator $B$.
More specifically, if $W$ is a Hilbert space, $Q$ is the $n-$fold product $Q = W\times W \times \dots \times W$, $A$ is an $n \times n$ matrix, and $B$ is a linear operator on $W$, then $A \otimes B$ is the linear operator on $Q$ defined by \eqref{eq:kronecker}.

Finally, we introduce a notation for uniform proportionality, used throughout the manuscript, 
as
\[
X \lesssim Y.
\]
That is, $X \lesssim Y$ implies the existence of some real constant 
$c_0 > 0$ such that $ X \leq c_0 Y$; any relationship between $c_0$ and pertinent 
mathematical objects, such as the total number of porous media networks considered, 
will be specified.

\subsection{Preconditioning of parameter-dependent systems}
\label{sec:preconditioning}

In this paper, we consider the preconditioning of discretizations of
the systems~\eqref{eq:mpet} and~\eqref{eq:mpt} under large parameter
variations. Therefore, we begin by summarizing core aspects of the
theory of parameter--robust preconditioning as presented in
\cite{MardalWinther2011}. We will apply this theory for formulations
of the MPT equations~\eqref{eq:mpt} and MPET equations~\eqref{eq:mpet}
in the subsequent sections.

Let $X$ be a separable, real Hilbert space with inner product
$\inner{\cdot}{\cdot}_X$, norm $\| \cdot \|_X$ and dual space
$X^*$.  Let $\mathcal{A}: X \rightarrow X$ be an invertible, symmetric
isomorphism on $X$ such that $\mathcal{A} \in \mathcal{L}(X,X^*)$
where $\mathcal{L}(X,X^*)$ is the set of bounded linear operators
mapping $X$ to its dual.  Given $f \in X^*$ consider the problem of
finding $x \in X$ such that
\begin{equation}
  \label{eq:linsys}
  \mathcal{A}x = f.
\end{equation}
The preconditioned problem reads as follows
\begin{equation}
\label{eq:preclinsys}
\mathcal{BA}x = \mathcal{B}f,
\end{equation}
where $\mathcal{B} \in \mathcal{L}(X^*,X)$ is a symmetric isomorphism
defining the preconditioner. The convergence rate of a Krylov space
method for this problem can be bounded in terms of the condition
number $\kappa(\mathcal{BA})$ where
\begin{equation*}
  \kappa(\mathcal{BA}) = \lVert \mathcal{BA} \rVert_{\mathcal{L}(X,X)} \lVert(\mathcal{BA})^{-1}\rVert_{\mathcal{L}(X,X)}.
\end{equation*}
Here, the operator norm $\lVert \mathcal{A} \rVert_{\mathcal{L}(X
  ,X^*)}$ is defined by
\begin{equation}
\label{eq:norm:operator}
\lVert \mathcal{A} \rVert_{\mathcal{L}(X, X^*)} = \sup_{x \in X} \frac{\lVert \mathcal{A} x \rVert_{X^*}}{\lVert x \rVert_{X}} .
\end{equation}

Now, for a parameter $\varepsilon$ (or more generally a set of
parameters $\varepsilon$) consider the parameter-dependent operator
$\mathcal{A}_\varepsilon$ and its preconditioner
$\mathcal{B}_\varepsilon$. Assume that we can choose appropriate
spaces $X_\varepsilon$ and $X^*_\varepsilon$ such that the norms
\begin{equation*}
  \|\mathcal{A}_\varepsilon \|_{\mathcal{L}(X_\varepsilon,X^*_\varepsilon)}
  \text{ and }
  \|\mathcal{A}^{-1}_\varepsilon \|_{\mathcal{L}(X^*_\varepsilon,X_\varepsilon)}
\end{equation*}
are bounded independently of $\varepsilon$. Similarly, we assume that
we can find a preconditioner $\mathcal{B}_\varepsilon$ such that the
norms $\lVert \mathcal{B}_\varepsilon \rVert_{
  \mathcal{L}(X_\varepsilon,X^*_\varepsilon)}$ and $\lVert
\mathcal{B}^{-1}_\varepsilon \rVert_{\mathcal{L}(X^*_\varepsilon,
  X_\varepsilon)}$ are bounded independently of $\varepsilon$. Given
these assumptions, the condition number $\kappa(\mathcal{B_\varepsilon
  A_\varepsilon})$ will be bounded independently of $\varepsilon$. We
will refer to such a preconditioner as robust in (or with respect to)
$\varepsilon$. %
We conclude this section with a change of variables result, recalled from 
basic matrix analysis \cite{horn1990matrix}, that will prove effective in the 
sections that follow.  

\begin{lemma}\label{lemma:congruence-change-of-variables}
Let $W$ be a real Hilbert space and $Q = W\times W\times \dots \times W$ be the 
$n$-fold direct product of $W$ for a fixed $n\in \mathbb{N}$.  %
Let $A: W \rightarrow W^{\ast}$, and $B : W \rightarrow W^{\ast}$ be linear 
operators. Suppose that $K,E \in \mathbb{R}^{n\times n}$ are symmetric matrices, and $K$ is invertible. 
Define the operators $S:Q\rightarrow Q^{\ast}$ and $T: Q\rightarrow Q^{\ast}$ 
by %
\begin{equation*}
S = K\otimes A,\quad \text{and}\quad T = E\otimes B,
\end{equation*}
where $\otimes$ is the Kronecker product. Consider the variational problem: given $f\in Q^*$ find $p=(p_1,p_2,\dots,p_n)^T \in Q$ such that  
\begin{equation}\label{eqn:lemma:congruence-change-of-variables:varform-a}
\inner{S p}{q} + \inner{T p}{q} = \inner{f}{q},\quad\forall\, q\in Q
\end{equation}
where $\inner{\cdot}{\cdot}$ is the duality pairing of $Q^*$ and $Q$. 
Then there exists an invertible matrix $P\in \mathbb{R}^{n\times n}$ such that 
the above variational problem is equivalent to: find $\tilde{p} \in Q$ such that
\begin{equation}\label{eqn:lemma:congruence-change-of-variables:varform-b}
\inner{D_S \tilde{p}}{q} + \inner{D_T \tilde{p}}{q} = \inner{F}{q},\quad\forall\,q\in Q,
\end{equation}
where $F = (P^T \otimes I_{W^*}) f$ for $I_{W^*}$ the identity operator on $W^*$, and 
$D_S = \left(P^T K P\right)\otimes A$ and $D_T = \left(P^T E P\right)\otimes B$ are block diagonal linear operators from $Q$ to $Q^*$.
\end{lemma}
\begin{proof}
Apply \cite[Theorem 4.5.17a-b p. 287]{horn1990matrix} the 
hypotheses on the matrices $K$ and $E$ and properties of the tensor product; %
see Appendix~\ref{appdx:lemma:congruence-change-of-variables} for detail.
\end{proof}

\section{Preconditioning the MPT equations via diagonalization}
\label{sec:mpt}

We begin by summarizing our novel approach to variational formulations
and associated preconditioning of the MPT equations. This approach was
introduced in \cite{PiersantiEtAlEnumath2019}. The core idea is to
reformulate the MPT (and MPET) equations using a change of pressure
variables $p$. In particular, we aim to find a transformation
of the variables $p \mapsto \tilde{p}$ such that the
transformed system of pressure equations decouple.  Here, we will consider 
a Hilbert space $W$ and the $J$-fold product $Q = W\times W \times\dots\times W$. 
Each pressure $p_j$, for $j=1,2,\dots,J$ satisfies $p_j \in W$ and we will  
write $p =(p_1, p_2, \dots, p_J)\in Q$.  In the sections that follow, we briefly
illustrate the core idea, formulation of the MPT equations and
resulting preconditioner, and refer to \cite{PiersantiEtAlEnumath2019}
for more details. This approach is then extended to the MPET equations
in Section~\ref{sec:mpet}.

\subsection{The MPT equations in operator form}

We consider the MPT equations as defined by~\eqref{eq:mpt}. We further
impose homogeneous Dirichlet boundary conditions for all pressures:
$p_j = 0$ on $\partial \Omega$ for $1 \leq j \leq J$.  
Define $\xi_{j} = \ssum_{i = 1}^J \xi_{j \leftarrow i}$ for each $1 \leq j
\leq J$. Let us define two $J \times J$ matrices:
\begin{equation}
  \label{eq:mpt:KD}
  K =
  \begin{pmatrix} K_1 & 0 & \cdots & 0 \\
    0 & K_2  & \cdots & 0 \\
    \vdots & \vdots & \ddots & \vdots \\ 
    0 & 0 & \cdots & K_J \\
  \end{pmatrix}, \quad
  E =
  \begin{pmatrix} \xi_1 & - \xi_{1 \leftarrow 2} & \cdots & - \xi_{1 \leftarrow J}  \\
    - \xi_{1 \leftarrow 2} & \xi_2 & \cdots & - \xi_{2 \leftarrow J} \\
    \vdots & \vdots & \ddots & \vdots \\ 
    - \xi_{1 \leftarrow J} & - \xi_{2 \leftarrow J} & \cdots & \xi_J \\
  \end{pmatrix} .
\end{equation}

The system~\eqref{eq:mpt} can be expressed in operator form as: given $g\in Q$ 
find $p\in Q$ satisfying  
\begin{equation}
  \label{eq:mpt:operator}
  \mathcal{A}_{\rm MPT} p = g \quad \text{ where } \quad  \mathcal{A}_{\rm MPT} = - K \otimes \Delta + E \otimes I_W. 
\end{equation}
In the above, $-K \otimes \Delta:Q \rightarrow Q^*$ is the block diagonal operator such that its $j$-th block is given by the bilinear form
$\inner{K_j \nabla p_j}{\nabla q_j}$ for $p_j,q_j \in Q_j=W$, and $E \otimes I_W:Q \rightarrow Q^*$ is the block operator such that its $(i,j)$-block $E_{ij}$ is defined by the bilinear forms
\begin{equation*}
  - \inner{\xi_{i\leftarrow j} p_i}{q_j} \text{  if } i \not = j, \qquad \inner{\xi_j p_j}{p_j} \text{  if } i = j .
\end{equation*}

We note that $K$ is real, positive definite and diagonal (and thus
invertible), and that $E$ is real, symmetric and (weakly row)
diagonally dominant by definition. In particular, $E$ is symmetric 
positive semi-definite because of the identity 
\begin{equation}
  \label{eq:E-pos-semidef}
 w E w^T = \sum_{1\le i,j \le J} \xi_{i \leftarrow j} (w_i - w_j)^2,
\end{equation}
for $w = (w_1, w_2, \ldots, w_J)$ with the convention $\xi_{i \leftarrow i} = 0$. %
A naive block diagonal preconditioner $\mathcal{B}_{\rm MPT}$ can be
constructed by taking the inverse of the diagonal blocks of
$\mathcal{A}_{\rm MPT}$. However, as we demonstrated in
\cite{PiersantiEtAlEnumath2019}, the resulting preconditioner is not
robust with respect to variations in the conductivity and exchange
parameters. In fact, the condition numbers increased linearly with
the ratio between the exchange and conductivity coefficients. 

\subsection{Change of variables using diagonalization by congruence}
\label{sec:mpt:transform}

In this section we discuss a new formulation for the MPT equations,
which in turn easily offers a parameter-robust preconditioner.  %
Let $P \in \R^{J \times J}$ be an invertible linear transformation
defining a change of variables and let $\tilde{p}$ and
$\tilde{q}$ be the new set of variables such that
\begin{equation}
\label{eq:pqtilde}
p = \left(P\otimes I_W\right) \tilde{p}, \quad q = \left(P\otimes I_W\right) \tilde{q},
\end{equation}
with $q = (q_1, q_2, \dots, q_J)$ and similarly for
$\tilde{q}, \tilde{p}$. %
Since $K$ and $E$ are symmetric, we apply Lemma~\ref{lemma:congruence-change-of-variables}, 
with $A = \Delta$ and $B=I_W$, to obtain a matrix, $P$, simultaneously diagonalizing 
$K$ and $E$ by congruence; that is, the equivalent operators 
$\left(P^T K P\right)\otimes \Delta$ and $(P^T E P)\otimes I_W$ are block diagonal. 
The resulting formulation (c.f.~\eqref{eqn:lemma:congruence-change-of-variables:varform-b}) of the 
MPT equations reads as follows: find the transformed pressures   

$\tilde{p} = (\tilde{p}_1, \dots, \tilde{p}_J)$ such that, for a given $g\in Q$, we have 
the equality
\begin{equation}
  \label{eq:mpt:transformed}
  \mathcal{\tilde{A}_{\rm MPT}} \tilde{p} 
  = (- \tilde{K} \otimes \Delta + \tilde{E} \otimes I_W ) \tilde{p}
  = \left(P^T\otimes I_W\right) g,
\end{equation}
where $\tilde{K} = P^T K P$ and $\tilde{E} = P^T E P$ are diagonal with
\begin{equation}
  \tilde{K} = \diag (\tilde{K}_1, \dots, \tilde{K}_J ), \quad
  \tilde{E} = \diag (\tilde{\xi}_1, \dots, \tilde{\xi}_J ).
\end{equation}

\subsection{Preconditioning the transformed MPT system}

As in \cite{PiersantiEtAlEnumath2019}, we can immediately identify the
parameter-dependent norm
\begin{equation*}
  \lVert \tilde{p} \rVert^2_{\tilde{\mathcal{B}}_{\rm MPT}}
  = \sum_{j=1}^J \inner{\tilde{K}_j \Grad \tilde{p}_j}{\Grad \tilde{p}_j} + \inner{\tilde{\xi}_j \tilde{p}_j}{\tilde{p}_j} 
\end{equation*}
and the following preconditioner associated to the above norm for~\eqref{eq:mpt:transformed}:
\begin{equation}
\label{eq:mpt:transformet:precond}
\mathcal{\tilde{B}}_{\rm MPT}
=
\begin{pmatrix} (-\tilde{K}_1 \Delta + \tilde{\xi}_1 I )^{-1} & 0 & \cdots & 0  \\
0 & (-\tilde{K}_2 \Delta + \tilde{\xi}_2 I)^{-1} & \cdots & 0\\
\vdots & \vdots & \ddots & \vdots \\ 
0 & 0 & \cdots & (-\tilde{K}_J \Delta + \tilde{\xi}_J I)^{-1} \\
\end{pmatrix} .
\end{equation}
Clearly, $\mathcal{\tilde{A}_{\rm MPT}}$ and
$\mathcal{\tilde{B}}^{-1}_{\rm MPT}$ are trivially spectrally
equivalent. We refer to \cite{PiersantiEtAlEnumath2019} for numerical
experiments comparing the standard and transformed formulation and
preconditioners.

\subsection{Finding the transformation matrix}
There are two cases that we will consider; the first case is when the matrix 
$C=K^{-1}E$ has $J$ distinct eigenvalues, while the second case will be for 
the case where at least one of the eigenvalues is repeated.  %
In the case of distinct eigenvalues, the number of distinct eigenvalues of $C = K^{-1} E$ will depend on
the material parameter values $K_j$ and $\xi_{j \rightarrow i}$ for $1
\leq i, j \leq J$. In the common case where $C$ has $J$ distinct
eigenvalues, the transformation matrix is easily defined as
follows. Let $\lambda_1, \dots, \lambda_J$ be the real eigenvalues of
$C$, and let $v_1, \dots, v_J$ be the corresponding normalized 
eigenvectors. Then,
\begin{equation}
  P = [v_1, \dots, v_J],
\end{equation}
will diagonalize $K$ and $E$ by congruence. In
\cite{PiersantiEtAlEnumath2019}, we presented numerical examples for
the case of $J$ distinct eigenvalues (with $J = 2$).

The congruence matrix for the case of repeated eigenvalues is also easily
constructed. For these cases, the transform $P$ can be constructed by
repeated application of block-wise eigenvector matrices, see
\cite{horn1990matrix} for the general procedure. In Example
\ref{ex:mpt:3net:multipleeig:diagonalization} below, we present an
example on how to obtain the transformation matrix $P$ in the case
where one of the eigenvalues has algebraic multiplicity $2$ with $J =
3$.

\begin{example}
\label{ex:mpt:3net:multipleeig:diagonalization}
In this example we show how to obtain the transformation matrix $P$
for a three--network case when one of the eigenvalues of $K^{-1} E$
has algebraic multiplicity $2$. In this example, due to the presence of 
the repeated eigenvector, the construction of $P$ does not follow directly 
from the use of normalized eigenvectors and, thus, $P$ is not normalized 
a priori.  We will, however, normalize $P$ following construction to maintain 
consistency with the previous case; in practice, either the normalized or 
non-normalized version of $P$ may be used.
\begin{equation}
  K =
  \begin{pmatrix}
    1.0 & 0 & 0 \\
    0 & 0.0001 & 0 \\
    0 & 0 & 0.01 \\
\end{pmatrix}, \quad
  E =
  \begin{pmatrix}
    1.01  & -0.01 & - 1.0  \\
    -0.01  & 0.0101 & -0.0001 \\
    -1.0  & -0.0001 & 1.0001
  \end{pmatrix} .
\end{equation} 
By definition
\begin{equation}
  C = K^{-1} E
  = \begin{pmatrix}
    1.01 & -0.01 & -1.0 \\
    -100 &  101 & -1.0 \\
    -100 & -0.01 &  100.01
  \end{pmatrix} .
\end{equation}
The eigenvalues $\lambda_1, \lambda_2, \lambda_3$ and eigenvectors
$[v_1, v_2, v_3] = P_1$ of $C$ are then:
\begin{equation}
  \begin{split}
    \lambda_1 &= 0, \quad \lambda_2 = \lambda_3 = 101.01; \\ 
    P_1 &= \begin{pmatrix} -0.5773 & -0.0071 & -0.0091\\ 
      -0.5773 &  0.7070 & -0.4031 \\
      -0.5773 &  0.7070 & 0.9150
    \end{pmatrix} .
  \end{split}
\end{equation}
In this specific case the eigenvalues $\lambda_2, \lambda_3$ have
algebraic multiplicity $2$ and geometrical multiplicity $1$. If we try
to diagonalize $K$ and $E$ by congruence via $P_1$, we obtain
\begin{equation}
\begin{split}
P_1^T K P_1 & = \begin{pmatrix} 3.3670 \times 10^{-1} & 0 & 0 \\
        0 & 5.1007 \times 10^{-3} & 6.5069 \times 10^{-3} \\
        0 & 6.5069 \times 10^{-3} & 8.4729 \times 10^{-3}
\end{pmatrix},\\ 
P_1^T E P_1 & = 101.01 \begin{pmatrix} 0 & 0 & 0 \\
        0 & 5.1007 \times 10^{-3} & 6.5069 \times 10^{-3} \\
        0 & 6.5069 \times 10^{-3} & 8.4729 \times 10^{-3}
\end{pmatrix}.
\end{split}
\end{equation}
In this case, the resulting matrices are block diagonal. The lower
right blocks are multiples of each other. We can diagonalize the lower
right blocks by computing the eigen-decomposition of either of
these. The lower right block of $P_1^T K P_1$ is
\begin{equation}
  \begin{pmatrix} 5.1007 \times 10^{-3} & 6.5069 \times 10^{-3} \\
    6.5069 \times 10^{-3} & 8.4729 \times 10^{-3}
  \end{pmatrix} 
\end{equation}
and its eigenpairs are
\begin{equation}
\begin{split}
\lambda_1 &= 6.4967 \times 10^{-5}, \lambda_2 =  1.3508 \times 10^{-2};\\ 
P_2 &= \begin{pmatrix} -0.79083 & -0.6120 \\
                        0.6120  & -0.7908 
\end{pmatrix}.
\end{split}
\end{equation}
The final transformation matrix $P$ that diagonalizes $K$ and $E$ by congruence is then:
\begin{equation}
  P = P_1 
  \begin{pmatrix}
    1 &
    \begin{matrix}
      0 & 0
    \end{matrix} \\
    \begin{matrix}
      0 \\ 0
    \end{matrix} & P_2
  \end{pmatrix} = 
  \begin{pmatrix}
    -5.7735 \times 10^{-1} &  7.1935 \times 10^{-5} &  1.1575 \times 10^{-2}\\
    -5.7735 \times 10^{-1} & -8.0594 \times 10^{-1} & -1.1391 \times 10^{-1}\\
    -5.7735 \times 10^{-1} &  8.6590\times 10^{-4} & -1.1564
\end{pmatrix}.
\end{equation}
Note that despite the columns of $P_1$ and $P_2$ are normalized with norm $1,$ the resulting matrix $P'$s columns are not normalized.
After the normalization, the matrix $P$ looks as follows:
\begin{equation}
  P = 
  \begin{pmatrix}
    -5.7735 \times 10^{-1} &  8.9255 \times 10^{-5} &  9.9611 \times 10^{-2}\\
    -5.7735 \times 10^{-1} & -9.9999 \times 10^{-1} & -9.8026 \times 10^{-2}\\
    -5.7735 \times 10^{-1} &  1.0743\times 10^{-4} & -9.9513 \times 10^{-1}
\end{pmatrix}.
\end{equation}
 and the diagonalized matrices are as follows
\begin{equation}
\begin{split}
  \tilde{K} = P^T K P &=
  \begin{pmatrix}
    3.3670 \times 10^{-1} & 0 & 0 \\
    0 &  1.0001 \times 10^{-4} & 0 \\
    0 &  0 &  1.0003\times 10^{-2}
  \end{pmatrix} , \\
  \tilde{E} = P^T E P &=
  \begin{pmatrix} 
    0 & 0 & 0 \\ 
    0 &  1.0102 \times 10^{-2} & 0 \\
    0 &  0 &  1.0104
  \end{pmatrix} .
\end{split}        
\end{equation}
\end{example}

\section{Preconditioning the MPET equations via diagonalization}
\label{sec:mpet}

In this section, we present a change of variables for the total
pressure formulation of the time-discrete MPET equations and propose
and analyze a preconditioning strategy for the resulting variational
formulation. The change of MPET variables is guided by the change of
MPT variables presented in the previous section.  As in Sec.~\ref{sec:mpt}, 
$Q$ will be defined as the $J$-fold product of a Hilbert space $W$.  

\subsection{Total pressure formulation of the MPET equations and its time-discrete form}

The total pressure formulation of Biot's equations~\cite{LeeEtAl2017}
and more generally the MPET equations~\cite{lee2019mixed} is a robust
mixed variational formulation targeting the nearly incompressible case
and incompressible limit ($\lambda \gg 1$). The total pressure, which we will 
see satisfies $p_0\in L^2(\Omega)$, is defined by
\begin{equation}
  \label{eq:def:p_0}
  p_0 = \lambda \Div u - \alpha \cdot p , 
\end{equation}
where 
\footnote{Note that we start counting at $1$ in the definition of $p$ here and 
throughout, in contrast to e.g.~in \cite{PiersantiEtAlEnumath2019}.} 
$\alpha = (\alpha_1, \dots,
\alpha_J)\in\mathbb{R}^J$, $p = (p_1, \dots, p_J)\in Q$ and $\alpha \cdot p =
\ssum_{i=1}^J \alpha_i p_i \in W$. The total pressure formulation
of~\eqref{eq:mpet} then reads as follows: for $t \in (0, T]$, find the
  displacement vector field $u$ and the pressure scalar fields $p_0$ and $p_j$
  for $j = 1, \dots, J$ such that
  \begin{subequations}
    \label{eq:mpet:tp}
    \begin{gather}
      \label{eq:mpet:tp1}
      - \Div \left ( 2 \mu \strain (u) + p_0 \Bbb{I} \right ) = f, \\
      \label{eq:mpet:tp2}
      \Div u - \lambda^{-1} p_0 - \lambda^{-1} \alpha \cdot p = 0, \\ 
      \lambda^{-1} \dot{p}_0 + s_j \dot{p}_j - \Div (K_j \nabla p_j) + \alpha_j \lambda^{-1} \alpha \cdot \dot{p}
      + \ssum_{i=1}^J \xi_{j \leftarrow i} (p_j - p_i) = g_j, 
    \end{gather}
\end{subequations}
for $j = 1, \dots, J$.
 
We consider an implicit Euler discretization in time of the total
pressure formulation of the time-dependent MPET
equations~\eqref{eq:mpet:tp} and examine the resulting stationary
problem at each time step. The resulting time-discrete version
of~\eqref{eq:mpet:tp} with time step $\tau > 0$ reads as follows: find
the displacement $u$ and the pressures $p_j$ for $0 \leq j \leq J$
such that
\begin{subequations}
  \label{eq:mpet:tp:semidiscretized}
  \begin{gather}
    - \Div \left ( 2 \mu \strain (u) + p_0 \Bbb{I} \right ) = f, \\ 
    \Div u - \lambda^{-1} p_0 - \lambda^{-1} \alpha \cdot p = 0, \\ 
    - s_j p_j - \alpha_j \lambda^{-1} p_0 - \alpha_j \lambda^{-1} \alpha \cdot p + \tau \Div (K_j \nabla p_j)
    - \tau \ssum_{i=1}^J \xi_{j \leftarrow i} (p_j - p_i) = g_j, 
  \end{gather}
\end{subequations} 
for $1 \leq j \leq J$ where the new right hand sides $g_j$ for $j=1,
\dots, J$ have been negated and contain also terms from the previous
time-step. Again, we impose homogeneous Dirichlet boundary conditions
for all network pressures: $p_j = 0$ on $\partial \Omega$ for $1 \leq
j \leq J$.

Let $V = H^1_{\Gamma}(\Omega; \R^d)$, $Q_0 = L^2(\Omega)$ and $Q_j = W =
H^1_0(\Omega)$ for $1 \leq j \leq J$ and $\Omega \subset \R^d$. Let $Q =
Q_1 \times \dots \times Q_J$. As in Section~\ref{sec:mpt}, we write $p
= (p_1, \dots, p_J)$, $q = (q_1, \ldots, q_J)$, and $g = (g_1, \ldots,
g_J)$.  Multiplying by test functions, and integrating second-order
derivatives by parts, we obtain the following variational formulation
of~\eqref{eq:mpet:tp:semidiscretized}: find $u \in V$ and $p_i \in
Q_i$ for $i = 0, \dots, J$ such that
\begin{subequations}
  \label{eq:mpet:vf}
  \begin{alignat}{2}
    a(u, v) + b(v, p_0) &= \inner{f}{v} &&\quad \foralls v \in V, \\
    b(u, q_0) - c_1(p_0, q_0) - c_2(q_0, p) &= 0 &&\quad \foralls q_0 \in Q_0, \\
    - c_2(p_0, q) - c_3(p, q) &= \inner{g}{q} &&\quad \foralls q \in Q.
  \end{alignat}
\end{subequations} 
The bilinear forms $a : V \times V \rightarrow \R$ and $b : V \times
Q_0 \rightarrow \R$ are defined as:
\begin{align}
  \label{eq:ab}
  a(u, v) & = \inner{2 \mu \strain(u)}{\strain(v)}, &  b(v, q_0) &=  \inner{\Div v}{q_0}, 
\end{align}
while $c_1 : Q_0 \times Q_0 \rightarrow \R$, $c_2 : Q_0 \times Q
\rightarrow \R$, and $c_3: Q \times Q \rightarrow \R$ are defined as:
\begin{align}
  \label{c1} c_1(p_0, q_0) &= \inner{\lambda^{-1} p_0}{q_0}, \\
  \label{c2} c_2(p_0, q) &= \inner{\lambda^{-1} \alpha \cdot q}{p_0} , \\
  \label{c3} c_3(p, q) &= \tau \sum_{j=1}^J \inner{K_j \Grad p_j}{\Grad q_j} + \sum_{j=1}^J \inner{s_j p_j}{q_j}, \\
 \notag &\quad + \tau \sum_{j=1}^J \sum_{i=1}^J \inner{\xi_{j \leftarrow i} (p_j - p_i)}{q_j} + \inner{\lambda^{-1} \alpha \cdot p}{\alpha \cdot q}.
\end{align}
For future reference we define $c : (Q_0 \times Q) \times (Q_0 \times Q) \rightarrow \R$ via
\begin{align} \label{c-form}
  c((p_0, p), (q_0, q)) &= c_1(p_0, q_0) + c_2(p_0, q) + c_2(q_0, p) + c_3(p, q) . 
\end{align}

\subsection{MPET as a parameter-dependent saddle point system}
The purpose of this subsection is to explain difficulties 
to construct parameter-robust block preconditioners 
for the system \eqref{eq:mpet:vf}. The difficulties lead 
us to develop the change of variables idea in the MPT problem in consideration of its extension to MPET problems.  

Note that the system~\eqref{eq:mpet:tp:semidiscretized} or
equivalently~\eqref{eq:mpet:vf} can be viewed as a
saddle point problem with a stabilization term (given by the bilinear form $c$). 
Thus, the equations fit well into Brezzi saddle point
theory~\cite{Brezzi1974}. However, various material parameters in different ranges are involved in the system, 
so constructing parameter-robust preconditioners for this system is not a straightforward application of the Brezzi theory. 
Let us recall the parameter ranges we are concerned with in this paper. 
The existing literature cover the parameters
\begin{equation*}
  0 \le s_j \lesssim 1, \qquad 0 < K_j \ll 1, \qquad  1 \lesssim \mu \lesssim \lambda < + \infty , \qquad (1 \le j \le J).
\end{equation*}
In addition to these we are also interested in developing preconditioners which are robust for the ratios of the exchange coefficients $\xi_{i \rightarrow j}$'s.  

We first consider construction of preconditioners utilizing the saddle point problem structure. 
To reveal the saddle point problem structure of~\eqref{eq:mpet:vf} let us look at the operator form of~\eqref{eq:mpet:vf}, which is 
\begin{equation}
  \label{eq:mpet:operator:form}
  \mathcal{A}_{\rm MPET}
  \begin{pmatrix}
    u \\
    p_0 \\
    p
  \end{pmatrix}
  = \begin{pmatrix} 
    - 2 \Div (\mu \varepsilon) & - \Grad & \textbf{0}\\ 
    \Div & - C_1 & - {C}_2^* \\
    \textbf{0} & - {C}_2 & - {C}_3 \\
  \end{pmatrix}
  \begin{pmatrix} u \\ p_0 \\ p
  \end{pmatrix}
  = \begin{pmatrix}
    f \\
    0 \\
    g \\
  \end{pmatrix}
\end{equation}
where $C_1: Q_0 \rightarrow Q_0^*$, $C_2: Q_0 \rightarrow Q^*$, ${C}_3: Q \rightarrow Q^*$ are the operators associated to the bilinear forms $c_1$, $c_2$, $c_3$ in \eqref{c1}, \eqref{c2}, \eqref{c3}. Here $C_2^*$ is the adjoint operator of $C_2$. 
We can rewrite $\mathcal{A}_{\rm MPET}$
of~\eqref{eq:mpet:operator:form} in the standard saddle point form
\begin{equation*}
  \mathcal{A}_{\rm MPET}
  = \begin{pmatrix} 
    A & B_0^* \\ 
    B_0 & -C
  \end{pmatrix}
\end{equation*}
by considering the product space grouping $V \times (Q_0 \times Q)$
and identifying
\begin{equation}
  \label{eq:AB}
  A = -2 \Div (\mu \varepsilon), \qquad B_0 = (\Div, \textbf{0})^T, \qquad  C
  =
  \begin{pmatrix}
    C_1 & C_2^* \\
    C_2 & C_3 \\
  \end{pmatrix} .
\end{equation}
One of natural approaches to construct block preconditioners for this system is 
to use the block diagonal operator 
\begin{equation*}
\begin{pmatrix}
  A^{-1} & 0 \\ 0 & (C + B_0 A^{-1} B_0^*)^{-1}
\end{pmatrix} 
\end{equation*}
or its approximation. However, the operator $(C + B_0 A^{-1} B_0^*)^{-1}$ is not easy to implement efficiently in practice.
Moreover, the analysis for spectral equivalence of this type of 
preconditioners is related to a non-trivial generalized eigenvalue problem. 
More precisely, the spectral equivalence is equivalent to uniform upper and lower bounds 
of the generalized eigenvalues, so it requires a deep analysis 
of the non-trivial generalized eigenvalue problem utilizing block matrix structures. 
In this paper we consider a general MPET model with general $J$ and general (constant) exchange coefficients, 
so the number of blocks in block matrices is not restricted. 
This makes an analysis of the generalized eigenvalue problem even more challenging, so we will not pursue this approach further in this paper.

Another natural choice of block preconditioners for this system is a direct extension of the
preconditioner in \cite{LeeEtAl2017}. In other words, we use the block diagonal operator of the form 
\begin{equation}
  \mathcal{B}_{\rm MPET} =
  \begin{pmatrix}
    (- \mu \Delta)^{-1} & 0 & 0 \\
    0 & I^{-1} & 0 \\
    0 & 0 & D^{-1} \\
  \end{pmatrix},  \;
  \label{eq:mpet:precond:naive}
\end{equation}
where $I: Q_0 \rightarrow Q_0^*$ is the operator defined by the bilinear form 
$\inner{p_0}{q_0}$ for $p_0, q_0 \in Q_0$, $D:Q \rightarrow Q^*$ is the block 
diagonal operator such that its $j$-th diagonal block ($1 \le j \le J$) is defined 
by selecting the $j^{\text{th}}$ diagonal entry of the operator $C_3$ associated to 
the bilinear form \eqref{c3}; that is  
\begin{equation*}
\tau \inner{K_j \Grad p_j}{\Grad q_j} + \inner{s_j p_j}{q_j} + \tau \inner{\xi_j p_j }{q_j} + \inner{\lambda^{-1} \alpha_j p_j}{\alpha_j q_j} , \qquad p, q \in Q.
\end{equation*}
However, this preconditioner is not robust with respect to the
material parameters, particularly for the hydraulic conductivity and the
exchange coefficients. We illustrate numerical experiment results in
Example~\ref{ex:mpet:tp}.

\begin{example}
  \label{ex:mpet:tp}
  Let $\Omega = [0, 1]^2 \subset \R^2$, and consider a structured
  triangulation $\mathcal{T}_h$ of $\Omega$ constructed by dividing
  $\Omega$ into $N \times N$ squares and then subdividing each square
  by a fixed diagonal. Let $J = 2$. Consider a finite element
  discretization of~\eqref{eq:mpet:vf} using the lowest order
  Taylor--Hood-type elements i.e.~continuous piecewise quadratics for
  each displacement component, and continuous piecewise linear for
  all pressures \cite{lee2019mixed}. Let $\tau = 1.0$, $\mu = 1.0$,
  $s_j = 1.0$, $\alpha_j = 0.5$ and $K_1 = 1.0$, and
  consider ranges of values for $\lambda, \xi_{1 \leftarrow 2}$ and
  $K_2$. We consider the case for $s_1=s_2 = 1.0$ and $s_1 = s_2 = 0.0.$
  Starting from an initial random guess, we consider a MinRes
  solver of the resulting linear system of equations with an algebraic
  multigrid (Hypre AMG) preconditioner of the
  form~\eqref{eq:mpet:precond:naive}.  
  The convergence criterion used was 
\[ 
 (\mathcal{B} r_k , r_k)/(\mathcal{B} r_0 , r_0) \leq 10^{-6} 
\]
 where $r_k$ is the residual of the $k$-th iteration. 
  The resulting number of
  Krylov iterations are shown in Table~\ref{tab:mpet:tp:naive} for
  $\xi_{1 \rightarrow 2} = 10^{6}$ and ranges of $K_2$ and
  $\lambda$. 
  We observe that the number of iterations is moderate ($\approx 30$)
  for $K_2$ of comparable magnitude ($10^6$) to $\xi_{1 \leftarrow
    2}$. The number of iterations increase with decreasing $K_2$: up
  to $\approx 1000$ for $K_2 = 1$. For large $K_2$, the number of
  iterations seems independent of the mesh resolution $N$. In
  contrast, for smaller $K_2$ (relative to $\xi_{1 \rightarrow 2}$), the
  number of iterations also increase with the mesh resolution. We note
  that the iteration counts do not vary substantially with $\lambda$. 
  \begin{table}
    \begin{center}
    \begin{tabular}{cc|ccccc}
      \toprule
      $K_2$     & $\lambda$      & $N=16$   & $32$   & $64$   & $128$  \\
      \midrule
       \multirow{4}{*}{$10^{0}$}     & $10^0$   & $738$  & $1271$ & $1756$ & $1938$ \\
                                      & $10^2$  & $1024$ & $1505$ & $1679$ & $1631$ \\
                                      & $10^4$  & $1028$ & $1506$ & $1666$ & $1628$ \\
                                      & $10^6$  & $1004$ & $1499$ & $1677$ & $1633$ \\
      \midrule
       \multirow{4}{*}{$10^{2}$}     & $10^0$  & $396$  & $424$  & $406$  & $353$  \\
                                      & $10^2$ & $337$  & $368$  & $351$  & $333$  \\
                                      & $10^4$ & $364$  & $352$  & $348$  & $332$  \\
                                      & $10^6$ & $345$  & $357$  & $361$  & $328$  \\
      \midrule
       \multirow{4}{*}{$10^{4}$}     & $10^0$  & $65$   & $65$   & $62$   & $60$   \\
                                      & $10^2$ & $64$   & $60$   & $56$   & $55$   \\
                                      & $10^4$ & $62$   & $60$   & $57$   & $55$   \\
                                      & $10^6$ & $63$   & $61$   & $58$   & $55$   \\
      \midrule
       \multirow{4}{*}{$10^{6}$}     & $10^0$   & $30$   & $30$   & $30$   & $28$   \\
                                      & $10^2$  & $34$   & $31$   & $29$   & $29$   \\
                                      & $10^4$  & $32$   & $31$   & $31$   & $29$   \\
                                      & $10^6$  & $33$   & $31$   & $31$   & $29$   \\
      \bottomrule
    \end{tabular}
    \caption{Number of MinRes iterations (c.f.~Example
        \ref{ex:mpet:tp}): \eqref{eq:mpet:vf} as discretized with
        Taylor-Hood type elements and an algebraic multigrid
        preconditioner of the form \eqref{eq:mpet:precond:naive}.  Of
        note is the fact that the number of iterations grow for $K_2$
      decreasing relative to $\xi_{2 \rightarrow 1} = 10^6$, and for
      increasing $N$.}
    \label{tab:mpet:tp:naive}
    \end{center}
  \end{table}
\end{example}

\subsection{Transformed MPET equations via change of variables}
\label{sec:mpet:transform}

In this subsection, we present MPET equations which are transformed via change of variables for 
construction of block preconditioners.  As in the MPT problem we will find an 
invertible linear map $P \in \R^{J \times J}$ that provides a fortuitous co-diagonalization; 
we will then consider the change of variables 
\begin{align*}
  p = \left(P\otimes I_W\right) \tilde{p},
\end{align*}
which will lead to a (partial) diagonalization, in the spirit of Lemma~\ref{lemma:congruence-change-of-variables}, 
for the transformed MPET system in the new unknowns $(u, p_0, \tilde{p})$.  %
For the discussions below let us give remarks on the block operators defined by $c_1$, $c_2$, $c_3$. 
Specifically, regarding $\alpha$ as a column vector, 
\begin{equation}
  \label{eq:def:C2}
  {C}_2 = \left(\lambda^{-1} \alpha\right) \otimes I_W , \quad
  {C}_3 =  - \tau K \otimes \Delta + \left(S + \tau E + L\right) \otimes I_W,
\end{equation}
with $K$ and $E$ as in~\eqref{eq:mpt:KD}, $L$ is the matrix $L_{ij} = \lambda^{-1}\alpha_{i}\alpha_{j}$, %
$S$ is the diagonal matrix such that its $j$-th entry is $s_j$, and $I_W$ is the 
identity (functional) on $W$; we recall that $Q$ is the $J$-fold cartesian 
product of $W$.  We will first describe the transformed MPET equations for general coordinate 
transformation, $P$. From the form of the transformed equations, we will extract 
the conditions for $P$ that yield a system that is suitable for the construction of  
parameter-robust block preconditioners.

Suppose we have an, fixed but otherwise arbitrary, invertible coordinate transformation matrix 
$P\in\mathbb{R}^{J\times J}$. Applying this transformation of variables to the semi-discretized
total pressure variational formulation of the MPET
equations~\eqref{eq:mpet:vf}, we obtain the following variational
formulation: find the displacement $u
\in V$, the total pressure $p_0 \in Q_0$ and the transformed pressures
$\tilde{p} = (\tilde{p}_1, \dots, \tilde{p}_J) \in Q$ such that
\begin{subequations}
  \label{eq:mpet:vf:transformed}
  \begin{alignat}{2}
	a(u, v) + b(v, p_0) &= \inner{f}{v} &&\quad \foralls v \in V, \\
    b(u, q_0) - c_1(p_0, q_0) - \tilde{c}_2(q_0, \tilde{p}) &= 0 &&\quad \foralls q_0 \in Q_0, \\
    - \tilde{c}_2(p_0, \tilde{q}) - \tilde{c}_3(\tilde{p}, \tilde{q}) &= \inner{g}{\left(P\otimes I_W\right) \tilde{q}} &&\quad \foralls \tilde{q} \in Q
  \end{alignat}
\end{subequations} 
where 
\begin{equation}
  \label{eq:def:tilde_c2_c3}
  \tilde{c}_2(q_0 , \tilde{q}) \equiv c_2(q_0, \left(P\otimes I_W\right) \tilde{q}), \qquad %
  \tilde{c}_3(\tilde{p}, \tilde{q}) \equiv c_3(\left(P\otimes I_W\right) \tilde{p}, \left(P\otimes I_W\right) \tilde{q}) .
\end{equation}
We define $\mathcal{\tilde{A}}_{\rm MPET}: V \times Q_0 \times Q \rightarrow (V \times Q_0 \times Q)^*$ as the operator corresponding to the bilinear form \eqref{eq:mpet:vf:transformed}. 
The operator form of the transformed system~\eqref{eq:mpet:vf:transformed} then reads as:
\begin{equation}
  \label{eq:mpet:tp:matrix:transformed}
  \mathcal{\tilde{A}}_{\rm MPET}
  \begin{pmatrix}
    u \\ p_0 \\ \tilde{p}
    \end{pmatrix} =
  \begin{pmatrix}
    f \\
    0 \\
    \tilde{g}
  \end{pmatrix},
  \quad
  \mathcal{\tilde{A}}_{\rm MPET}
  =  
  \begin{pmatrix} 
    A & B^T & \textbf{0} \\ 
    B & - C_1 & - \tilde{C_2}^* \\
    \textbf{0} & - \tilde{C_2} & - \tilde{C_3}
  \end{pmatrix} ,
\end{equation}
where $A = - 2 \Div( \mu \varepsilon)$, $B = \Div$
as before, and
$\tilde{g} = \left(P^T\otimes I_W\right) g$. By inserting~\eqref{eq:def:C2} and reordering, we
note that
\begin{align*}
  \tilde{C_2} &= P^T C_2  =   (\lambda^{-1} P^T \alpha) \otimes I_W ,
  \\
  \tilde{C_3} &= P^T C_3 P  
  = - \tau (P^T K P) \otimes \Delta + \left(P^T S P + P^T \left (\tau E + L \right ) P\right) \otimes I_W .
\end{align*}
For simplicity we will write
\begin{equation}
  \tilde \alpha = \left(P^T\otimes I_W\right) \alpha.
\end{equation}
We now look to apply Lemma~\ref{lemma:congruence-change-of-variables} with the choice 
of operators $S=K\otimes \Delta$ and $T = \left(\tau E + L\right)\otimes I_W$.  The 
matrices $K$ and $\tau E + L$ satisfy the required conditions and, thus,  
there exists (c.f.~Appendix~\ref{appdx:lemma:congruence-change-of-variables}) an invertible 
transformation $P$ simultaneously diagonalizing $K$ and $\tau E + L$ by congruence.  
That is, we have matrices $\tilde{K}$ and $\tilde{\Gamma}$ given by the formulas  
\begin{align}
  \label{eq:transformedK}
  \tilde{K} &= P^T K P = \diag ( \tilde{K}_1, \dots, \tilde{K}_J ), \\
  \label{eq:transformedGamma}
  \tilde{\Gamma} &= P^T \left (\tau E + L \right ) P = \diag ( \tilde{\gamma}_1, \dots, \tilde{\gamma}_J ). 
\end{align}
We point out that the storage coefficients $\{ s_j \}_{j=1}^J$ are not involved in this simultaneous diagonalization process.
This is critically important in order to achieve a preconditioner that is parameter-robust, 
even in the presence of vanishing storage coefficients.  For future reference we briefly note that 
\begin{equation}
  \label{eq:tilde:gamma}
  \tilde \gamma_j \ge  \tilde{\alpha}_j^2/\lambda, 
\end{equation}
follows from the definition of $\tilde{\Gamma}$ in \eqref{eq:transformedGamma} since $E$ 
is positive semi-definite and therefore $\tilde{\gamma}_j$ is greater than or equal to 
the $j$-th diagonal entry of the matrix 
$\left(P^T L P\right)_{ij} = \lambda^{-1} \tilde{\alpha}_i \tilde{\alpha}_j$.

We also remark that the following identity holds for $\tilde{c}_3$: 
\begin{align} \label{eq:tilde_c3_identity}
 \tilde{c}_3(\tilde p, \tilde q) &= \tau \sum_{j=1}^J \inner{\tilde{K}_j \Grad \tilde p_j}{\Grad \tilde q_j} %
  + \sum_{j=1}^J \inner{s_j \left(\left(P\otimes I_W\right) \tilde p)\right)_j}{\left(\left(P\otimes I_W\right) \tilde q\right)_j}, \\
  \notag &\quad + \tau \sum_{j=1}^J \sum_{i=1}^J \inner{\xi_{j \leftarrow i} ( \left(\left(P\otimes I_W\right) \tilde p\right)_j - \left(\left(P\otimes I_W\right) \tilde p\right)_i)}{\left(\left(P\otimes I_W\right) \tilde q\right)_j} \\%
 \notag&\quad + \inner{\lambda^{-1} \tilde{\alpha} \cdot \tilde p}{\tilde{\alpha} \cdot \tilde q}
\end{align}
where $\left(\left(P\otimes I_W\right)\tilde{p}\right)_j$ is the $j$-th component of $\left(P\otimes I_W\right) \tilde{p}$.

\subsection{Preconditioning of the transformed MPET system}
\label{sec:MPET:spectral}
In this subsection we show that a parameter-robust preconditioner can be constructed using an appropriate parameter-dependent norm.

We first define a parameter-dependent norm
\begin{equation}
  \label{eq:B:norm}
  \| (u, p_0, \tilde p) \|_{\tilde{\mathcal{B}}}^2
   = \| \varepsilon(u) \|_{2 \mu}^2
   + \| p_0 \|_{(2 \mu)^{-1}}^2
   + \sum_{j=1}^J \| \Grad \tilde p_j \|_{\tau \tilde K_j}^2
   + \sum_{j=1}^J \|  \tilde p_j \|_{\tilde{\gamma}_j}^2 
\end{equation}
and consider the associated block preconditioner of the form 
\begin{equation}
  \label{eq:mpet:preconditioner:transformed}
  \mathcal{\tilde{B}}_{\rm MPET} =
  \begin{pmatrix} (- 2 \mu \Delta)^{-1} & 0 & 0 & \cdots & 0 \\
    0 & 2 \mu I^{-1} &  0 & \cdots & 0 \\
    0 & 0 & (- \tau \tilde{K}_1 \Delta + \tilde{\gamma}_1 I )^{-1}  & \cdots & 0 \\
    \vdots & \vdots & \vdots & \ddots & \vdots \\
    0 & 0 &  0 & \cdots & (- \tau \tilde{K}_J \Delta + \tilde{\gamma}_J I )^{-1} 
  \end{pmatrix}.
\end{equation}

\begin{lemma}[Continuity]
  \label{lem:continuity}
  Let $\tilde{\mathcal{A}}_{\rm MPET}$ be defined by
  \eqref{eq:mpet:tp:matrix:transformed}, and consider the norm defined
  by~\eqref{eq:B:norm}. We assume that $2\mu \le M_0 \lambda$ for some $M_0 >0$ 
  and $s_j \lesssim \tilde{\gamma}_j$ for $1 \le j \le J$. 
  Then there exists a constant $C>0$, dependent on $M_0$, the constants in $s_j \lesssim \tilde{\gamma}_j$, the matrix $P$, and 
  the number of networks $J$ but independent of any other problem 
  parameters, such that
  \begin{equation}
    \inner{\tilde{\mathcal{A}}_{\rm  MPET} (u, p_0, \tilde p)}{(v, q_0, \tilde q)} 
    \leq C \| (u, p_0, \tilde p) \|_{\tilde{\mathcal{B}}} \| (v, q_0, \tilde q) \|_{\tilde{\mathcal{B}}},
  \end{equation}
  for all $(u, p_0, \tilde p), (v, q_0, \tilde q) \in V \times Q_0 \times Q$. 
\end{lemma}
\begin{proof} 
  By redistributing the material parameter weights and
  the Cauchy-Schwarz inequality, we obtain the preliminary upper bound
\begin{equation*}
  \begin{split}
    \inner{\tilde{\mathcal{A}}_{\rm  MPET} (u, p_0, \tilde p)}{(v, q_0, \tilde q)} 
    \leq Z_1 + Z_2 + Z_3 =: Z,
  \end{split}
\end{equation*}
where
\begin{align*}
  Z_1 &= \| \varepsilon(u) \|_{2 \mu} \| \varepsilon(v) \|_{2 \mu} + \| p_0 \|_{(2 \mu)^{-1}} \| \Div v \|_{2 \mu}
  + \| q_0 \|_{(2 \mu)^{-1}} \| \Div u \|_{2 \mu}, \\
  Z_2 &= \| p_0 \|_{\lambda^{-1}} \| q_0 \|_{\lambda^{-1}} + \| \tilde \alpha \cdot \tilde q \|_{\lambda^{-1}} \| p_0 \|_{\lambda^{-1}}
  + \| \tilde \alpha \cdot \tilde p \|_{\lambda^{-1}} \| q_0 \|_{\lambda^{-1}}, \\ 
  Z_3 &= \ssum_{j=1}^J \left ( \| \Grad \tilde p_j \|_{\tau \tilde K_j} \| \Grad \tilde q_j \|_{\tau \tilde K_j} 
     + \| \tilde p_j \|_{\tilde \gamma_j} \| \tilde q_j \|_{\tilde \gamma_j} + \inner{s_j (P \tilde{p})_j}{(P \tilde{q})_j} \right ).
\end{align*}
Since $\| \Div u \| \leq \| \varepsilon(u) \|$ and by $2\mu \le  M_0 \lambda$ and the assumptions on $s_j$ and $\tilde{\gamma}_j$, it follows that
\begin{equation*}
  \begin{split}
    Z \lesssim &\left ( \| \varepsilon(u) \|_{2 \mu} + \| p_0 \|_{(2 \mu)^{-1}} + \| \tilde \alpha \cdot \tilde p \|_{\lambda^{-1}}
    + \ssum_{j=1}^J \left ( \| \Grad \tilde p_j \|_{\tau \tilde K_j} + \| \tilde p_j \|_{\tilde \gamma_j} \right )
    \right ) \\
    &\times \left ( \| \varepsilon(v) \|_{2 \mu} + \| q_0 \|_{(2 \mu)^{-1}} +  \| \tilde \alpha \cdot \tilde q \|_{\lambda^{-1}}
    + \ssum_{j=1}^J \left ( \| \Grad \tilde q_j \|_{\tau \tilde K_j} + \| \tilde q_j \|_{\tilde \gamma_j} \right )
    \right ) .
  \end{split}
\end{equation*}
By the triangle inequality and \eqref{eq:tilde:gamma} we obtain 
\begin{equation}
\| \tilde \alpha \cdot \tilde p \|_{\lambda^{-1}} \le \sum_{j=1}^J \| \tilde{p}_j \|_{\tilde{\alpha}_j^2 / \lambda} \le \sum_{j=1}^J \| \tilde p_j \|_{\tilde \gamma_j} 
\end{equation}
and it completes the proof.
\end{proof}

\begin{lemma}[Inf-sup condition] \label{lem:inf-sup}
  Let $\tilde{\mathcal{A}}_{\rm MPET}$, $\tilde{\mathcal{B}}_{\rm
    MPET}$ and all assumptions be as in
  Lemma~\ref{lem:continuity}. Then, there exists a constant $C > 0$,
  dependent on $M_0$, the constants in $s_j \lesssim \tilde{\gamma}_j$, and the number of networks $J$ but independent of other
  material parameters, such that
\begin{equation}
  \label{eq:inf-sup}
  \adjustlimits
  \inf_{(u,p_0, p)}
  \sup_{(v,q_0, q)}
  \frac{\inner{\tilde{\mathcal{A}}_{\rm MPET}(u, p_0, p)}{(v, q_0, q)}}{\|(u, p_0, p)\|_{\tilde{\mathcal{B}}} \| (v, q_0, q)\|_{\tilde{\mathcal{B}}}} \ge C ,
\end{equation}
where the $\inf$ and $\sup$ are taken over the non-vanishing elements in $V \times Q_0 \times Q$. 
\end{lemma}
\begin{proof} 
  Consider any $(u, p_0, \tilde p) \in V \times Q_0 \times Q$, and choose $\tilde q
  = - \tilde p$, and $q_0 = - p_0$. Let $w \in V$ satisfy 
  \begin{equation}
    \label{eq:assumpt:w}
    \inner{\Div w}{p_0} = \|p_0\|_{(2\mu)^{-1}}^2, \quad \| \varepsilon(w) \|_{2\mu} \leq C_0 \|p_0\|_{(2\mu)^{-1}} .
  \end{equation}
  for a $C_0 > 0$ depending on the domain $\Omega$ via Korn's
  inequality, and next choose $v = u + 2\delta w$ for $\delta > 0$ to
  be further specified. We note that, with this choice of $v, q_0$,
  and $q$,
  \begin{equation*}
    \| (v, q_0, \tilde q) \|_{\tilde{\mathcal{B}}} \lesssim \| (u, p_0, \tilde p) \|_{\tilde{\mathcal{B}}},
  \end{equation*}
  with inequality constant depending only on the domain $\Omega$ and the choice of $\delta$ since
  \begin{equation*}
    \| \varepsilon(v) \|_{2 \mu} \leq \| \varepsilon(u) \|_{2 \mu} + 2\delta C_0 \| p_0 \|_{(2 \mu)^{-1}} \lesssim \| (u, p_0, \tilde p) \|_{\tilde{\mathcal{B}}}.
  \end{equation*}
  Therefore, it suffices to show that
  \begin{equation}
    \label{eq:lower:bound}
    \inner{\tilde{\mathcal{A}}_{\rm MPET}(u, p_0, \tilde p)}{(v, q_0, \tilde q)}
    \gtrsim  \| (u, p_0, \tilde p) \|_{\tilde{\mathcal{B}}}^2.
  \end{equation}
  
  Using the definition of $\tilde{\mathcal{A}}_{\rm MPET}$ together
  with~\eqref{eq:assumpt:w}, we find that
  \begin{multline}
    \label{eq:A_mpet_lower}
      \inner{\tilde{\mathcal{A}}_{\rm MPET}(u, p_0, \tilde p)}{(v, q_0, \tilde q)}
      = \| \varepsilon(u) \|^2_{2\mu} + 2\delta \inner{\varepsilon(u)}{\varepsilon(w)}_{2 \mu} + 2\delta \| p_0 \|^2_{(2 \mu)^{-1}} \\
      + c_1(p_0, p_0) + 2 \tilde{c}_2(p_0, \tilde{p}) + \tilde{c}_3(\tilde{p}, \tilde{p}) .
  \end{multline}
  Note that $c_1(p_0, p_0) = \| p_0\|^2_{\lambda^{-1}}$, $\tilde{c}_2(p_0, \tilde{p}) = \inner{\tilde \alpha \cdot \tilde p}{p_0}_{\lambda^{-1}}$, and 
  \begin{align*}
    \tilde{c}_3(\tilde{p}, \tilde{p}) = \ssum_{j=1}^J (\| \Grad \tilde p_j \|^2_{\tau \tilde K_j} +  \| \tilde p_j \|^2_{\tilde \gamma_j} + \inner{s_j (P \tilde{p})_j }{(P \tilde{p})_j} )
  \end{align*}
  from the definitions of $\tilde{c}_2$, $\tilde{c}_3$, and the congruent diagonalization. Moreover, we have 
  \begin{align}
    \label{eq:tmp-ineq}
    \ssum_{j=1}^J \| \tilde p_j \|^2_{\tilde \gamma_j} \ge \| \tilde{\alpha} \cdot \tilde p \|^2_{\lambda^{-1}} 
  \end{align}
  by observing the identity \eqref{eq:tilde_c3_identity}. Then we obtain
  \begin{align*}
    &  \delta \| p_0 \|_{(2\mu)^{-1}}^2  +  c_1(p_0, p_0) + 2 \tilde{c}_2( p_0, \tilde{p}) + \tilde{c}_3( \tilde{p}, \tilde{p})  \\
    &\ge \delta \| p_0 \|_{(2\mu)^{-1}}^2  +  \| p_0 \|_{\lambda^{-1}}^2 + 2 \inner{\tilde \alpha \cdot \tilde p}{p_0}_{\lambda^{-1}} + \ssum_{j=1}^J (\| \tilde p_j \|_{\tilde \gamma_j}^2 + \| \Grad \tilde p_j \|^2_{\tau \tilde K_j})  \\
    &\ge  \left( {\delta}/{M_0} + 1 \right) \| p_0 \|_{\lambda^{-1}}^2  + 2 \inner{\tilde \alpha \cdot \tilde p}{p_0}_{\lambda^{-1}} + \ssum_{j=1}^J (\| \tilde p_j \|_{\tilde \gamma_j}^2 + \| \Grad \tilde p_j \|^2_{\tau \tilde K_j}) \\
    &=  \| \sqrt{{\delta}/{M_0} + 1 } p_0 + \tilde \alpha \cdot  \tilde p \|_{\lambda^{-1}}^2 - \left(\frac{\delta}{M_0} + 1 \right)^{-1} \| \tilde{\alpha} \cdot \tilde{p} \|_{\lambda^{-1}}^2  \\
    &\quad + \ssum_{j=1}^J (\| \tilde p_j \|^2_{\tilde{\gamma}_j} + \| \Grad \tilde p_j \|^2_{\tau \tilde K_j}) \\
    &\ge   \ssum_{j=1}^J \left( \frac{\delta}{M_0 + \delta} \| \tilde p_j \|^2_{\tilde{\gamma}_j} +  \| \Grad \tilde p_j \|^2_{\tau \tilde K_j} \right)  
  \end{align*}
  where the second inequality in the above follows from the assumption $2 \mu \le M_0 \lambda$ and the third inequality follows from \eqref{eq:tmp-ineq}. 
  Thus,
  \begin{align}
    \label{eq:infsup:1}
    \notag &\inner{\tilde{\mathcal{A}}_{\rm MPET}(u, p_0, \tilde p)}{(v, q_0, \tilde q)} \\
      &\quad \geq \| \varepsilon(u) \|^2_{2\mu} + \delta \inner{\varepsilon(u)}{\varepsilon(w)}_{2 \mu} + \delta \| p_0 \|^2_{(2 \mu)^{-1}} + \ssum_{j=1}^J\| \Grad \tilde p_j \|^2_{\tau \tilde K_j} \\
    \notag &\qquad + \frac{\delta}{M_0 + \delta} \ssum_{j=1}^J  \| \tilde p_j \|^2_{\tilde{\gamma}_j}.
  \end{align}
  On the other hand, the Cauchy-Schwarz inequality, the definition of
  $w$, and Young's inequality give that
  \begin{equation}
    \label{eq:delta}
    \delta |\inner{\varepsilon(u)}{\varepsilon (w)}_{2 \mu}|
    \leq \delta C_0 \| \varepsilon(u)\|_{2 \mu} \| p_0\|_{(2 \mu)^{-1}}
    \leq \frac{1}{2} \| \varepsilon(u)\|^2_{2 \mu} + \frac{1}{2} \delta^2 C_0^2 \| p_0\|^2_{(2 \mu)^{-1}} .
  \end{equation}
  Inserting the negation of~\eqref{eq:delta} as a lower bound
  in~\eqref{eq:infsup:1}, we thus obtain that
  \begin{multline*}
      \inner{\tilde{\mathcal{A}}_{\rm MPET}(u, p_0, \tilde p)}{(v, q_0, \tilde q)} \\ 
      \geq \frac{1}{2} \| \varepsilon(u) \|^2_{2\mu} + \delta (1 - \frac{1}{2} \delta C_0^2) \| p_0 \|^2_{(2 \mu)^{-1}} 
      + \ssum_{j=1}^J \left( \| \Grad \tilde p_j \|^2_{\tau \tilde K_j} +  \frac{\delta}{M_0 + \delta} \| \tilde p_j \|^2_{\tilde{\gamma}_j} \right).
  \end{multline*}
  By choosing $\delta$, in particular e.g.~by letting $\delta <
  2/C_0^2$, the estimate \eqref{eq:lower:bound} follows.
\end{proof}

\begin{rmk}
  In Lemma~\ref{lem:continuity} and Lemma~\ref{lem:inf-sup} we assumed $s_j \lesssim \tilde{\gamma}_j$ for $1 \le j \le J$ and it covers the cases that $s_j$'s are degenerate. This assumption can be removed if $\lambda^{-1} \lesssim s_j$, $1 \le j \le J$, hold with constants of scale 1. For parameter-robust preconditioners we use $P$ which gives a different simultaneous diagonalization. More precisely, we consider $P$ satisfying \eqref{eq:transformedK}, and $P^T (S + \tau E + L) P = \diag ( \tilde{\gamma}_1, \dots, \tilde{\gamma}_J )$ instead of \eqref{eq:transformedGamma}. The norm \eqref{eq:B:norm} with these new $\tilde{K}_j$'s and $\tilde{\gamma}_j$'s, will be used to obtain parameter-robust preconditioners. Since the modification of proofs is straightforward and most steps are almost same, we omit the detailed proofs. 
\end{rmk}

For concreteness, we here illustrate the form of the MPET equations
and of the proposed preconditioner in a specific example. 
\begin{example}
  We consider the simple case of two networks with $K_1 = K_2 = 1.0$,
  $s_1 = s_2 = 1.0$, $\alpha_1 = \alpha_2 = 0.5$, $\lambda = 1.0$,
  $\xi_{1 \rightarrow 2} = 0.0$, and $\tau = 1.0$. The transformation
  matrix in this case is
  \begin{equation}
    P = \frac{1}{\sqrt{2}} \begin{pmatrix}
      1 & -1 \\
      1 & 1
    \end{pmatrix}.
  \end{equation}
  We remark that $P$ is not normalized. 
  The associated transformed MPET operator (expanded),
  cf.~\eqref{eq:mpet:tp:matrix:transformed} and associated
  definitions, is then
  \begin{equation}
    \tilde{\mathcal{A}}_{\rm MPET} =
    \begin{pmatrix}
      - 2 \mu \Div \varepsilon & - \Grad & 0 & 0 \\
      \Div & - \lambda^{-1} & - (\sqrt 2 \lambda)^{-1} &  0 \\
      0 & - (\sqrt 2 \lambda)^{-1} & - \Delta + \frac{3}{2}  &   0\\
      0 & 0 &  0 & - \Delta + 1\\
    \end{pmatrix},
  \end{equation}
  and the proposed preconditioner will be in the following form:
  \begin{equation}
    \mathcal{\tilde{B}}_{\rm MPET} =
    \begin{pmatrix} (- 2 \mu \Div \varepsilon)^{-1} & 0 & 0 & 0 \\
      0 & (2 \mu)^{-1} &  0 & 0 \\
      0 & 0 & (- \Delta + 1)^{-1}  &     0\\
      0 & 0 &  0 & (- \Delta + 1)^{-1}
    \end{pmatrix}.
  \end{equation}
The objective of this example was to illustrate the layout of the operators in a simple case. 
The results for more general numerical examples will be presented later.
  \end{example}

\subsection{Numerical performance}

\begin{example}
  \label{ex:mpet:performance}
  In this example we demonstrate the robustness of the block
  diagonal preconditioner \eqref{eq:mpet:preconditioner:transformed}
  for a mixed finite element discretization of the transformed total
  pressure MPET equations~\eqref{eq:mpet:tp:matrix:transformed}.  We
  consider the same test case, discretization and solver set-up as
  described in Example~\ref{ex:mpet:tp}; the new preconditioner is the
  only modification. Parameter ranges are as follows: $K_2
    \in[10^{-6},10^{6}]$, $\xi_{1\leftarrow 2} \in [10^{-6},10^{6}]$
    and $\lambda \in [1,10^{6}]$. 

The resulting number of iterations are shown in Figure \ref{fig:boat3}
for $K_2 \in [10^{-6},1]$ and $\xi_{1\leftarrow 2} \in [1,10^{6}]$;
omitted values demonstrated similar behaviour. Each of the subplots in
Figure \ref{fig:boat3} represent a fixed choice of $K_2$ and
$\xi_{1\leftarrow 2}$. In each subplot four curves are shown; these
curves show the number of MinRes iterations corresponding to different
values of $\lambda$, indicated by their respective symbols, at
discretization levels $N=16$, $32$, $64$ and $128$.  The stopping
criterion was
\[ 
(\tilde{\mathcal{B}} r_k , r_k)/(\tilde{\mathcal{B}} r_0, r_0) \leq 10^{-6}
\] 
where $r_k$ is the residual of the $k$-th iteration. We observe that
the number of iterations is moderate in general. Moreover, the number
of iterations does not grow for smaller $K_2$'s relative to larger
$\xi_{1 \leftarrow 2}$ or larger $N$ -- in contrast to what was
observed for Example \ref{ex:mpet:tp}. 
\end{example}

\begin{figure}
  \includegraphics[width=\textwidth]{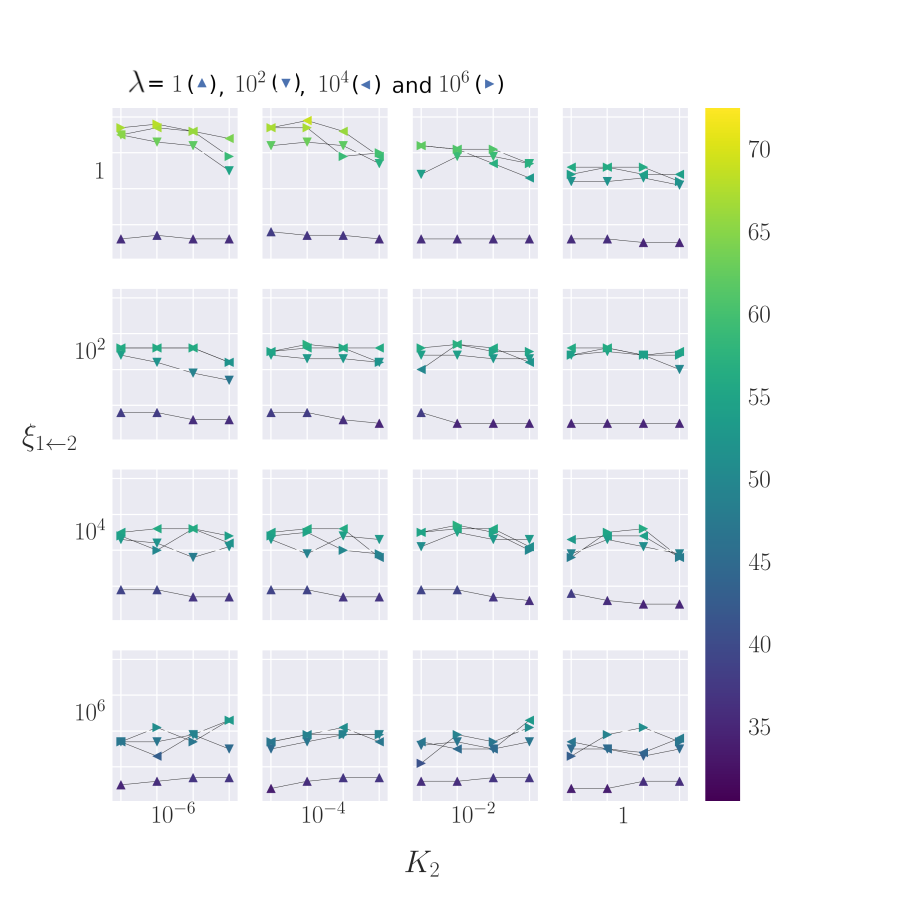}
\caption{Number of MinRes iterations: \eqref{eq:mpet:vf} discretized with 
Taylor-Hood type elements and algebraic multigrid, for $s_1=s_2=1.0$.  $K_2$ varies along the horizontal 
axis while the vertical axis shows variations in $\xi_{1\leftarrow 2}$ for 
$K_2$ fixed.  Each subplot contains four piecewise linear curves; each curve is 
decorated by a symbol indicating a corresponding value of $\lambda$ and corresponds 
to results for discretizations $N=16,32,64$ and $128$.}
  \label{fig:boat3}
\end{figure}

\begin{figure}
  \includegraphics[width=\textwidth]{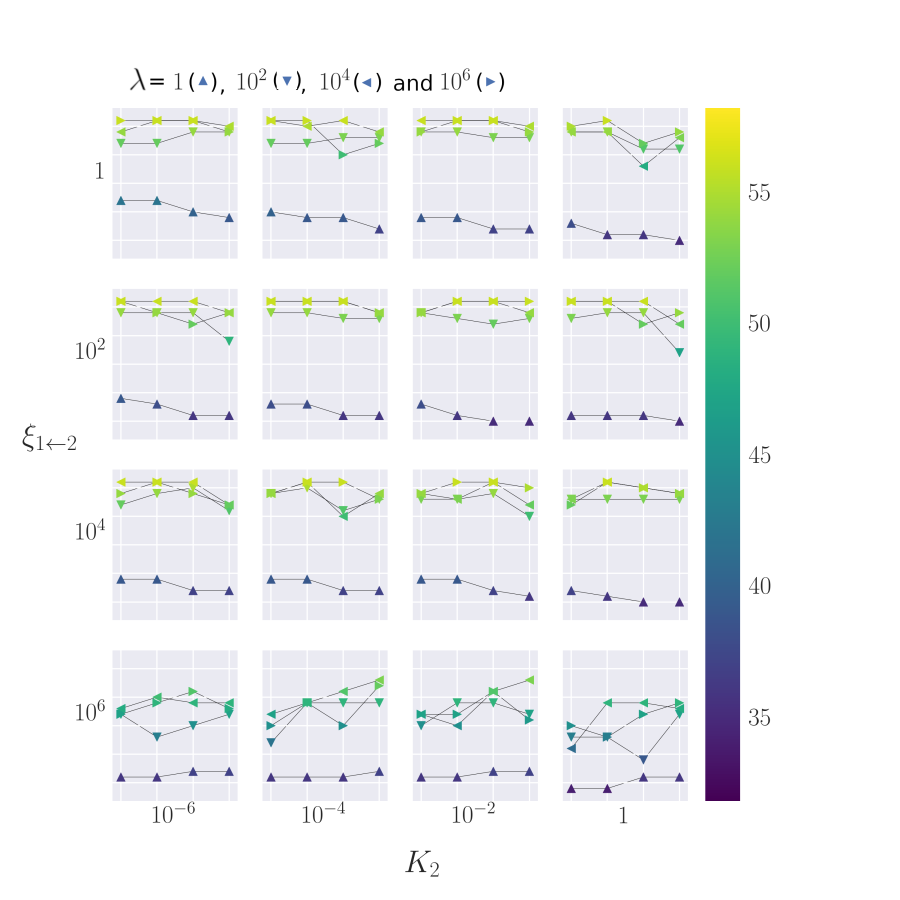}
\caption{Number of MinRes iterations: \eqref{eq:mpet:vf} discretized with 
Taylor-Hood type elements and algebraic multigrid for $s_1 = s_2 =  0$.  $K_2$ varies along the horizontal 
axis while the vertical axis shows variations in $\xi_{1\leftarrow 2}$ for 
$K_2$ fixed.  Each subplot contains four piecewise linear curves; each curve is 
decorated by a symbol indicating a corresponding value of $\lambda$ and corresponds 
to results for discretizations $N=16,32,64$ and $128$.}
  \label{fig:boat4}
\end{figure}

\begin{example}
\label{ex:footing-problem}
In this final example we present a modified version of a 3D footing problem 
\cite{gaspar2008,rodrigo2018,storvik2019}; we demonstrate the problem for two 
fluid networks ($J=2$) and use the standard unit cube, $\Omega = [0,1]^3 \subset \R^3$, 
as computational domain.  At the base of the domain, homogeneous Dirichlet conditions 
for the displacement and for both fluid pressures are imposed.At the top-most surface 
of the domain, i.e.~$z=1$, a load of $0.1 N/m^2$ is applied on the square 
$[0.25,0.75]\times[0.25,0.75]$, and a no flow condition is applied to the fluid 
pressures. For all remaining boundary sides of the domain, the zero stress 
condition is applied alongside a homogeneous Dirichlet condition for the fluid 
pressures.  In the numerical experiments we vary the exchange coefficient 
$\xi_{1\leftarrow 2},$ and the mesh size, the other physical parameters are 
reported in Table~\ref{tab:footingproblem:parameters}.
\begin{figure}[ht]
\centering
\begin{subfigure}{.32\textwidth}
  \centering
  \includegraphics[width=\linewidth]{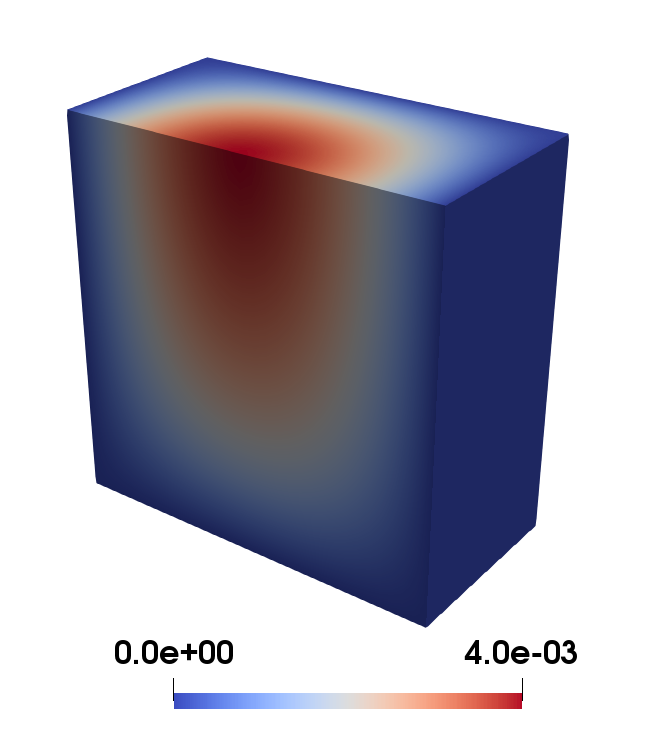}  
  \caption{fluid pressure $p_1$}
  \label{fig:sub-first}
\end{subfigure}
\begin{subfigure}{.32\textwidth}
  \centering
  \includegraphics[width=\textwidth]{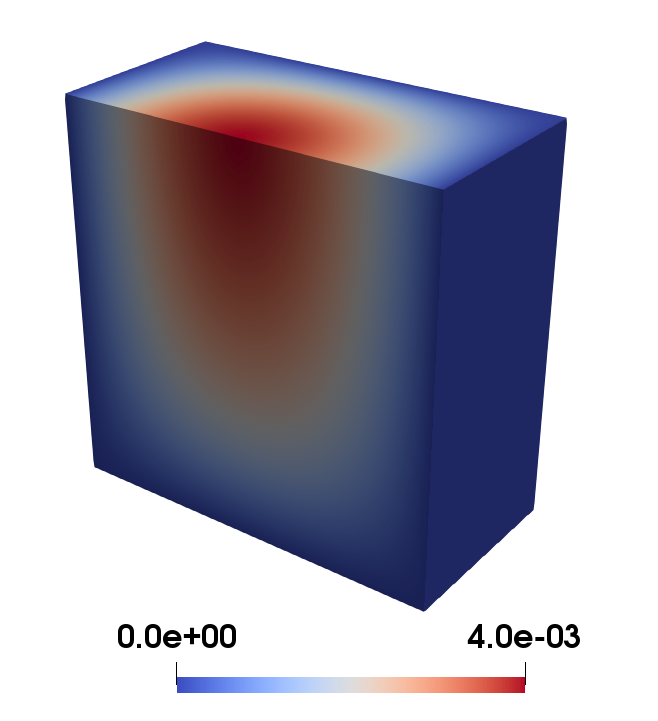}  
  \caption{fluid pressure $p_2$}
  \label{fig:sub-second}
\end{subfigure}
\begin{subfigure}{.32\textwidth}
  \centering
  \includegraphics[width=\textwidth]{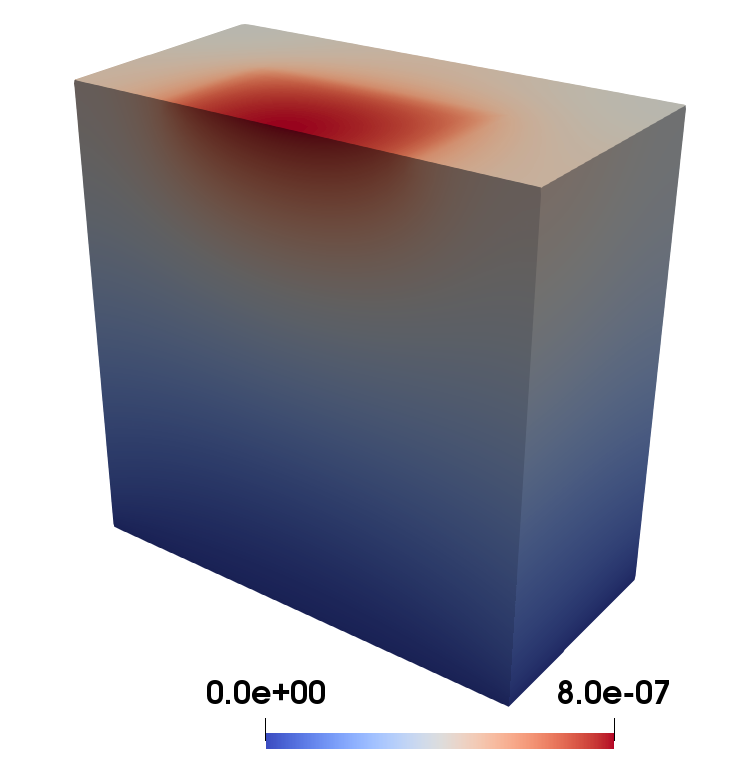}  
  \caption{displacement magnitude $u$}
  \label{fig:sub-third}
\end{subfigure}
\hfill
\caption{A 3D Footing Problem, solution for $1/h = 32,$ $\xi_{1\leftarrow 2} = 1.0,$ at $t= 0.5$}
\label{fig:footing-solution}
\end{figure}

\begin{table}[h!]
\centering
\begin{tabular}{lllll}
 Property & Symbol &  Value & Units  \\
 \hline
 Young's modulus & $E$ & $3\times 10^4$ &  Pa  \\
 Poisson ratio & $\nu$ & 0.45 & [-]  \\
 Hydraulic conductivities & $K_1, K_2$ & $10^{-6}$ &  m$^2$(Pa s)$^{-1}$\\
 Storage coefficients & $s_1, s_2$ & $0.0$ & Pa$^{-1}$ \\
 Biot coefficient & $\alpha_1, \alpha_2$ & $0.5$ & [-] \\
 \hline
\end{tabular}
\caption{Parameters used in the numerical simulations}
\label{tab:footingproblem:parameters}
\end{table}

In table~\ref{tab:footingproblem:iterations}, we report the number of MinRes iterations for each time-step (from $0.1$ to $0.5$), varying the mesh size and exchange parameters. The initial guess for the solution is set to zero. Similarly to what observed in Example \ref{ex:mpet:performance}, the number of iterations is moderate also for this 3D case. 
In Fig. \ref{fig:footing-solution} the solution for $t=0.5$ is shown.
\begin{table}[ht]
\centering
\begin{tabular}{l|l|l|l|l|l|l}
    \hline
    \multirow{2}{*}{h} & \multirow{2}{*}{$\xi_{1 \leftarrow 2}$} & \multicolumn{5}{c}{Number of iterations}\\
    \cline{3-7}
    & &$t=0.1$ & $t=0.2$ & $t=0.3$ & $t=0.4$ & $t=0.5$ \\ 
    \hline
    \multirow{2}{*}{$1/8$} & $1.0 \times 10^{-6}$& $87$ & $97$& $97$ & $97$ & $97$ \\
						   & $1.0$ & $89$ & $102$ & $102$ & $102$ & $102$ \\
	\hline					   
    \multirow{2}{*}{$1/16$} & $1.0 \times 10^{-6}$ & $90$ & $102$ & $102$ & $102$ & $102$ \\
						   & $1.0$ & $93$ & $108$ & $109$ & $107$ & $109$ \\
	\hline    
    \multirow{2}{*}{$1/32$} & $1.0 \times 10^{-6}$& $95$ & $107$ & $107$ & $107$ & $107$\\
						    & $1.0$ & $98$ & $112$ & $112$ & $114$ & $111$ \\
	\hline
\end{tabular}
\caption{MinRes iterations for the footing problem (c.f.~Example~\ref{ex:footing-problem}).}
\label{tab:footingproblem:iterations}
\end{table}
 \end{example}

\section{Conclusions}
\label{sec:conclusion}

In this paper, we have presented a new strategy for decoupling the
total-pressure variational formulation of the multiple-network
poroelasticity equations. The decoupling strategy is based on a
transformation via a change of variables, allowing for simultaneous
diagonalization by congruence of the equation operators. In
particular, the transformed equations are readily amenable for
block--diagonal preconditioning. Moreover, we have proposed a
block-diagonal preconditioner for the transformed system and shown
theoretically that the preconditioner and the equation operator are
norm equivalent, independently of the material parameters, under
reasonable parameter assumptions. The theoretical results are
supported by numerical examples. Combined, these results allow the
efficient iterative solution of the multiple-network poroelasticity
equations, even in the case of nearly incompressible materials.

We note that our strategy is based on spatially constant material
parameters. The applicability of this approach for spatially varying
parameters has not yet been considered.

\appendix
\section{Proof of Lemma \ref{lemma:congruence-change-of-variables}}
\label{appdx:lemma:congruence-change-of-variables}
We first recall a basic \cite{horn1990matrix} definition and result for posterity.
\begin{definition} A matrix $C\in\mathbb{C}^{n\times n}$ is \textit{diagonalizable} if 
there exists an invertible transformation, $P$, such that $P^{-1} C P$ is diagonal. 
The matrix $C$ is called \textit{diagonalizable by congruence} if there exists $P$, 
not necessarily invertible, such that $P^T C P$ is diagonal.  
\end{definition}

\begin{thm}[4.5.17a-b p. 287, \cite{horn1990matrix}]
Suppose $A$ and $B\in \mathbb{C}^{n\times n}$ are symmetric and that $A$ is invertible. 
Then $A$ and $B$ are diagonalizable by congruence if and only if $C = A^{-1}B$ is 
diagonalizable.  
\end{thm}

\begin{proof}[Proof of Lemma~\ref{lemma:congruence-change-of-variables}]
Assume $K$ and $E$ satisfy the hypotheses of %
Lemma~\ref{lemma:congruence-change-of-variables};  we first show that 
$C = K^{-1}E \in \mathbb{R}^{n\times n}$ is diagonalizable.  We note that $C$ satisfies 
\[
K^{1/2} C K^{-1/2} = K^{-1/2} E K^{-1/2}. 
\]
The right-hand side, above, is symmetric due to the symmetry of $K$ and $E$. 
Thus $C$ is similar to a real, symmetric matrix and is therefore diagonalizable.  
From \cite[4.5.17a-b]{horn1990matrix} there exists an invertible matrix $P \in \mathbb{R}^{n\times n}$ 
such that 
\[
P^T K P = \tilde{D}_K\quad\text{ and }\quad  P^T E P = \tilde{D}_E,
\]
where $\tilde{D}_K$, $\tilde{D}_E\in\mathbb{R}^{n\times n}$ are diagonal matrices.  

Recalling $Q = W \times \cdots \times W$, define the change of variables 
$\tilde{p} = (P^{-1} \otimes I_W) p$ for $p\in Q$ and substitute into 
\eqref{eqn:lemma:congruence-change-of-variables:varform-a} to get 
\begin{equation*}
\inner{S(P \otimes I_W) \tilde{p}}{q} + \inner{T(P \otimes I_W)\tilde{p}}{q} = \inner{f}{q},\quad \forall\,q\in Q.
\end{equation*} 
Writing $q = (P \otimes I_W) (P^{-1} \otimes I_W) q$ and noting that the adjoint operator of $P \otimes I_W$ is $P^T \otimes I_{W^*}$, we have 
\begin{equation*}
\inner{S(P \otimes I_W) \tilde{p}}{q} = \inner{(P^T \otimes I_{W^*}) S (P \otimes I_W) \tilde{p}}{(P^{-1} \otimes I_W) q}.  
\end{equation*}
Since $S = K \otimes A$ we can obtain $(P^T \otimes I_{W^*}) S (P \otimes I_W) = \tilde{D}_K \otimes A$ (by Hadamard product). 
By a similar argument 
\[
\inner{TP\tilde{p}}{q} = \inner{(\tilde{D}_E\otimes B)\tilde{p}}{(P^{-1} \otimes I_W ) q}.
\]
Then $D_S := \tilde{D}_K \otimes A$ and $D_T := \tilde{D}_E \otimes B$ are block diagonal 
operators from $Q$ to $Q^*$.  Finally $\inner{f}{q} = \inner{(P^T \otimes I_{W^*}) f}{(P^{-1} \otimes I_W) q}$ and the variational 
problem \eqref{eqn:lemma:congruence-change-of-variables:varform-b} follows because 
$q\in Q$ is arbitrary.  
\end{proof}

\begin{rmk}
The construction of the matrix $P$, yielding both $P^T A P = D_1$ 
and $P^T B P = D_2$, is straightforward for the case when $C=A^{-1}B$ has $n$ 
distinct eigenvalues.  In this case $C$ has $n$ linearly independent 
eigenvectors; if $\{v_1, \dots, v_n\}$ denote these eigenvectors then 
$P = [v_1, \dots, v_n]$ is the matrix whose $j$-th column is $v_j$.  
When the eigenvalues of $C$ are not distinct: $P$ can be realized as the product 
of block-wise eigenvector matrices.  The general procedure for this case is 
discussed in \cite{horn1990matrix}; an example has been discussed in 
Section~\ref{sec:mpt}.
\end{rmk}

\bibliographystyle{siamplain}
\bibliography{references}

\end{document}